\documentclass{amsart}
\newcommand{\mt}[1]{\mathtt{#1}}

\newcommand{\mF}{\mathcal{F}}

\newcommand{\bn}{\bar{n}}
\newcommand{\bB}{\mathbb{B}}

\usepackage{multirow}
\usepackage{booktabs,caption}
\usepackage{amsmath}
\usepackage{amsthm}
\usepackage{graphicx}
\usepackage{color}
\usepackage{amsfonts}
\usepackage{amssymb}
\usepackage{mathrsfs}
\usepackage{amscd}
\usepackage{url}

\newcommand{\on}{\bar{n}}
\newcommand{\re}{\mathbb{R}}
\newcommand{\cpx}{\mathbb{C}}

\newcommand{\N}{\mathbb{N}}

\newcommand{\lmd}{\lambda}

\newcommand{\eps}{\epsilon}

\def\af{\alpha}
\def\bt{\beta}
\def\gm{\gamma}
\def\rank{\mbox{rank}}

\newcommand{\sig}{\sigma}
\newcommand{\Sig}{\Sigma}

\newcommand{\reff}[1]{(\ref{#1})}

\newcommand{\mc}[1]{\mathcal{#1}}

\newcommand{\bdes}{\begin{description}}
\newcommand{\edes}{\end{description}}

\newcommand{\bal}{\begin{align}}
\newcommand{\eal}{\end{align}}

\newcommand{\bnum}{\begin{enumerate}}
\newcommand{\enum}{\end{enumerate}}

\newcommand{\bit}{\begin{itemize}}
\newcommand{\eit}{\end{itemize}}

\newcommand{\bea}{\begin{eqnarray}}
\newcommand{\eea}{\end{eqnarray}}
\newcommand{\be}{\begin{equation}}
\newcommand{\ee}{\end{equation}}

\newcommand{\baray}{\begin{array}}
\newcommand{\earay}{\end{array}}

\newcommand{\bsry}{\begin{subarray}}
\newcommand{\esry}{\end{subarray}}

\newcommand{\bca}{\begin{cases}}
\newcommand{\eca}{\end{cases}}

\newcommand{\bcen}{\begin{center}}
\newcommand{\ecen}{\end{center}}

\newcommand{\bbm}{\begin{bmatrix}}
\newcommand{\ebm}{\end{bmatrix}}

\newcommand{\bmx}{\begin{matrix}}
\newcommand{\emx}{\end{matrix}}

\newcommand{\bpm}{\begin{pmatrix}}
\newcommand{\epm}{\end{pmatrix}}

\newcommand{\btab}{\begin{tabular}}
\newcommand{\etab}{\end{tabular}}

\newtheorem{theorem}{Theorem}[section]

\newtheorem{lem}[theorem]{Lemma}

\newtheorem{cor}[theorem]{Corollary}

\theoremstyle{definition}

\newtheorem{exm}[theorem]{Example}
\newtheorem{alg}[theorem]{Algorithm}

\begin{document}

\title[Low Rank Symmetric Tensor Approximations]
{Low Rank Symmetric Tensor Approximations}

\author{Jiawang Nie}
\address{
Department of Mathematics,  University of California San Diego,  9500
Gilman Drive,  La Jolla,  California 92093,  USA.
}
\email{njw@math.ucsd.edu}

\begin{abstract}
For a given symmetric tensor, we aim at finding a new one
whose symmetric rank is small and that is close to the given one.
There exist linear relations among the entries of low rank symmetric tensors.
Such linear relations can be expressed by polynomials,
which are called generating polynomials.
We propose a new approach for computing low rank approximations
by using generating polynomials.
First, we estimate a set of generating polynomials
that are approximately satisfied by the given tensor.
Second, we find approximate common zeros
of these polynomials. Third, we use these zeros to
construct low rank tensor approximations.
If the symmetric tensor to be approximated is sufficiently
close to a low rank one, we show that the computed
low rank approximations are quasi-optimal.
%
%This approach can also be applied to
%efficiently compute low rank decompositions for symmetric tensors,
%especially for large scale tensors.
%
\end{abstract}



\maketitle

\section{Introduction}

Let $m,n >0$ be integers. Denote by
$\mt{T}^m(\cpx^n)$ the space of $m$th order tensors
over the complex vector space $\cpx^n$.
Under the canonical basis of $\cpx^n$,
each $\mc{F} \in \mt{T}^m(\cpx^n)$ can be represented by
an array indexed by integer tuples $(i_1, \ldots, i_m)$ with
$1 \leq i_j \leq n \, (j=1,\ldots,m)$, i.e.,
\[
\mc{F} = (\mc{F}_{i_1 \ldots i_m})_{1 \leq i_1, \ldots,  i_m \leq n }.
\]
Tensors of order $m$ are called $m$-tensors.
When $m=3$ (resp., $4$), they are called cubic (resp., quartic) tensors.
The tensor $\mF$ is {\it symmetric} if
$\mF_{i_1 \ldots i_m}$ is invariant under all permutations of $(i_1, \ldots, i_m)$.
Denote by $\mt{S}^m(\cpx^n)$
the linear subspace of all symmetric tensors in $\mt{T}^m(\cpx^n)$.

For a vector $u \in \cpx^n$, its $m$th tensor power is
the outer product $u^{\otimes m} \in \mt{S}^m(\cpx^n)$ such that
it holds for all $1 \leq i_1 \ldots i_m \leq n$ that
\[
( u^{\otimes m} )_{i_1 \ldots i_m}
= u_{i_1} \cdots u_{i_m}.
\]
Tensors like $u^{\otimes m}$ are called rank-1 symmetric tensors.
For every $\mc{F} \in \mt{S}^m(\cpx^n)$, there exist vectors
$u_1, \ldots, u_r \in \cpx^n$ such that
\be \label{dcmp:rank-r}
\mc{F} = (u_1)^{\otimes m} + \cdots + (u_r)^{\otimes m}.
\ee
The smallest such $r$ is called the {\it symmetric rank} of $\mF$,
and is denoted as $\rank_S(\mF)$.
If $\rank_S(\mF) =r$, $\mF$ is called a rank-$r$ tensor and
\reff{dcmp:rank-r} is called a symmetric rank decomposition,
which is often called a {\it Waring decomposition} in the literature.
The rank of a generic symmetric tensor is given
by the Alexander-Hirschowitz formula \cite{AlxHirs95}.
We refer to \cite{CGLM08} for symmetric tensors and their symmetric ranks,
and refer to \cite{BalBer12,BerGimIda11,BCMT10,GPSTD,OedOtt13}
for symmetric tensor decompositions. On the other hand,
the general {\it rank} of $\mF$, denoted as $\rank(\mF)$,
is the smallest $k$ such that
\[
\mc{F} = \mF_1 + \cdots + \mF_k,
\]
where each $\mF_i \in \mt{T}^m(\cpx^n)$ is rank-$1$
but not necessarily symmetric.
In the literature, $\rank(\mF)$
is also called the candecomp-parafac (cp) rank of $\mF$.
Clearly, we always have
$\rank(\mF) \leq \rank_S(\mF).$
It is interesting to know whether or not $\rank(\mF) = \rank_S(\mF)$
for a symmetric tensor $\mF$.
Comon conjectured that they are equal \cite{CConj}.
Indeed, Friedland \cite{Fri15} proved this is true for some classes of tensors.
Throughout this paper, we only consider symmetric tensors,
and their symmetric ranks are just called ranks, for convenience.

Tensor decomposition is a fundamental question
in multilinear algebra. For nonsymmetric tensors,
there exist optimization based methods
for computing tensor decompositions. We refer to
Acar et al.~\cite{ADKM}, Comon et al.~\cite{CLdA09}, Hayashi et al.~\cite{HayHay},
Paatero~\cite{Paat}, Phan et al.~\cite{PTC13},
Sorber et al.~\cite{SBL13}, Tomasi and Bro~\cite{TomBro}.
These methods can be adapted to computing symmetric tensor decompositions.
A survey of tensor decomposition methods can be found in
Comon~\cite{Com00}, Kolda and Bader~\cite{KolBad09}.
For symmetric tensor decompositions, there exist methods
that are based on catalecticant matrices \cite{IaKa99},
Hankel matrices and flat extensions \cite{BCMT10},
optimization based methods~\cite{Kolda15},
tensor eigenvectors \cite{OedOtt13},
and generating polynomials~\cite{GPSTD}.
Moreover, the method in \cite{BCMT10} can be generalized 
to compute nonsymmetric tensor decompositions~\cite{BBCM13}.
Tensor decompositions have broad applications \cite{KolBad09}.
For symmetric tensors, Waring decompositions have
broad applications, for instance, in machine learning \cite{AGHKT}.
For more introductions about tensors, we refer to \cite{Land12,Lim13}.

In applications, people often need to approximate tensors by low rank ones.
A symmetric tensor in $\mt{S}^m(\cpx^n)$ has $n^m$ entries,
which grows rapidly as $n$ and/or $m$ increase.
%
%and the space $\mt{S}^m(\cpx^n)$ has dimension $\binom{n+m-1}{m}$.
%
This number is big, even for small $m$ and moderately large $n$.
For instance, when $n=100$ and $m=3$, a cubic tensor has one million entries.
Computations with tensors are often expensive, because of the typical huge dimension.
So, low rank approximations are often preferable
in applications~\cite{ComLim11,GKT13}.
A rank-1 tensor in $\mt{S}^m(\cpx^n)$ can be parameterized as $u^{\otimes m}$,
i.e., by an $n$-dimensional vector $u$,
which is significantly smaller than $n^m$.
Therefore, if we can find $u_1, \ldots, u_r$ such that
\be \label{lra:F=u1r}
\mc{F} \, \approx \, (u_1)^{\otimes m} + \cdots + (u_r)^{\otimes m},
\ee
then the computations with $\mc{F}$
can be approximately done with $u_1,\ldots,u_r$.
The low rank symmetric tensor approximation problem is the following:
for a given tensor $\mc{F} \in \mt{S}^m(\cpx^n)$
and a given rank $r$ (typically small),
find vectors $u_1, \ldots, u_r \in \cpx^n$
such that \reff{lra:F=u1r} is satisfied as much as possible.
This is equivalent to the nonlinear nonconvex optimization problem
\be \label{nLS:lra:symF}
\min_{ u_1,\ldots, u_r \in \cpx^n } \quad
\Big\| (u_1)^{\otimes m} + \cdots +(u_r)^{\otimes m}
- \mc{F} \Big\|^2,
\ee
where the Hilbert-Schmidt norm as in \reff{tensor:norm} is used.
The problem~\reff{nLS:lra:symF} is NP-hard~\cite{HiLi13},
even for the special case $r=1$.

In contrast to the matrix case (i.e., $m=2$), the best rank-$r$
tensor approximation may not exist when $m>2$ and $r>1$~\cite{DeSLim08}.
This is because the set of tensors of rank less than
or equal to $r$ may not be closed.
For $r=1$, there exists much work on rank-$1$ approximations;
see, e.g., \cite{LMV00b,HNZ17,SHOPM2002,NieWan13,ZhaGol01,ZLQ12}.
For almost all $\mF$, the best rank-$1$
approximation is unique \cite{FriOtt14}.
%
%Kofidis and Regalia \cite{SHOPM2002},
%Hu, Huang and Qi~\cite{HHQ13}, Zhang, Ling and Qi \cite{ZLQ12},
%Nie and Wang \cite{NieWan13}.
%
For $r>1$, when $\mF$ is a nonsymmetric tensor,
there exists various work on rank-$r$ approximations.
The classical methods \cite{ADKM,CLdA09,HayHay,Paat,PTC13,SBL13,TomBro}
%e.g., alternative least squares,
%and higher order power iterations \cite{???LMV00b},
are often used for nonsymmetric tensor decompositions and approximations.
We refer to \cite{ComLim11,FriTam14,GKT13}
for recent work on nonsymmetric low rank tensor approximations.
When $r>1$ and $\mF$ is symmetric,
there exists relatively few work on computing low rank approximations.
The classical methods for the nonsymmetric case
could be used by forcing the symmetry in the computation~\cite{SHOPM2002},
but their theoretical properties are not well-studied.

%
%De Lathauwer et al. \cite{LMV00b} investigated best
% rank-$(R_1,R_2,\ldots,R_N)$ approximations
%for higher order tensors. There are
%quasi-Newton type methods (cf.~\cite{SavLim10})
% and Jacobi rotation type methods (cf. \cite{IAvD13})
%for computing such approximations.
%

\bigskip \noindent
{\bf Contributions}\,
In this paper, we propose a new approach for computing
low rank approximations for symmetric tensors.
It is motivated by the existing linear relations
among the entries of low rank symmetric tensors.
Such linear relations can be expressed by polynomials,
which are called generating polynomials~\cite{GPSTD}.
In applications of low rank approximations,
the tensor to be approximated is often close to a low rank one.
This is typically the case because of measurement errors or noise,
which are often small.
This fact motivates us to compute low rank approximations
by finding the hidden linear relations that are
satisfied by low rank symmetric tensors.

Our method of computing low rank symmetric tensor approximations
consists of three major stages.
First, we estimate a set of generating polynomials,
which can be obtained by solving a linear least squares problem.
Second, we find approximate common zeros
of these generating polynomials;
it can be done by computing Schur decompositions and solving eigenvalue problems.
Third, we construct a low rank approximation from their common zeros,
by solving a linear least squares problem.
Our main conclusion is that if the tensor to be approximated is sufficiently close
to a low rank one, then the computed low rank tensors are good low rank approximations.
The proof is build on perturbation analysis of linear least squares
and Schur decompositions.
%
%The proposed methods
%can also be applied to compute low rank decompositions
%efficiently, especially for large scale tensors.
%

The paper is organized as follows.
In Section~\ref{sc:prlm}, we present some basics for symmetric tensors.
In Section~\ref{sc:lra:sym}, we give an algorithm for computing
low rank symmetric tensor approximations,
and then analyze its approximation quality in Section~\ref{sbsc:symerr}.
In Section~\ref{sc:comp}, we report numerical experiments.
Section~\ref{sc:confu} concludes the paper and lists some future work.

\section{Preliminaries}
\label{sc:prlm}
\setcounter{equation}{0}

\subsection{Notation}

The symbol $\N$ (resp., $\re$, $\cpx$) denotes the set of
nonnegative integers (resp., real, complex numbers).
For any $t\in \re$, $\lceil t\rceil$ (resp., $\lfloor t\rfloor$)
denotes the smallest integer not smaller
(resp., the largest integer not bigger) than $t$.
%For integer $n>0$, $[n]$ denotes the set $\{1,\ldots,n\}$.
%and $\supp{\af} = \{ 1\leq i\leq n:\, \af_i \ne 0\}$.
%For $\af,\bt \in \N^n$, denote $\af \leq  \bt$ if every $\af_i \leq \bt_i$.
%The symbol $\N_{\leq k}^n$ denotes the multi-index set $\{\af\in\N^n: |\af| \leq k \}$,
%%and $\N_{=k}^n$ denotes $\{\af\in\N^n: |\af| = k \}$.
%For $x \in \cpx^n$ and $\af \in \N^n$, $x^\af$
%denotes the monomial $x_1^{\af_1}\cdots x_n^{\af_n}$.
%%The symbol $[x]_d$ denotes the following vector of monomials
%\[
%[x]_d^T = [\, 1 \quad  x_1 \quad \cdots \quad x_n \quad x_1^2 \quad
%x_1x_2 \quad \cdots \cdots
%\quad x_1^d \quad x_1^{d-1}x_2 \quad \cdots \cdots \quad x_n^d \,],
%\]
%
%The symbol $\cpx[x] := \cpx[x_1,\ldots,x_n]$)
%denotes the ring of complex polynomials in $x:=(x_1,\ldots,x_n)$.
%The ring $\re[x] := \re[x_1,\ldots,x_n]$)
%is defined similarly for real polynomials.
%
The cardinality of a set $S$ is denoted as $|S|$.
%
%For a finite set $\mathbb{B}$ in $\cpx[x]$ and a vector $v \in \cpx^n$, denote
%\be \label{v:to:*m}
%[v]_\mathbb{B}  := \big( p(v) \big)_{p \in \mathbb{B}}, \quad
%v^{\circledast m} := (1, v)^{\otimes m}.
%\ee
%%
%The $[x^d]$ denotes the column vector of all monomials of degree $d$, i.e.,
%\[
%[x^d]^T  := [\, x_1^d \quad x_1^{d-1}x_2 \quad \cdots \cdots \quad x_n^d \,].
%\]
%%
%A homogeneous polynomial is called a form.
%The $\cpx[x]_d$ denotes the subspace of complex polynomials
%with degree $\leq d$.
%%
%For a finite set $S$, $|S|$ denotes its cardinality.
%For a general set $S \subseteq \re^n$,
%$int(S)$ denotes its interior, and $\pt S$ denotes its boundary in standard Euclidean topology.
%
For a complex matrix $A$, $A^T$ denotes its transpose,
$A^*$ denotes its conjugate transpose,
$\| A \|_2$ denotes its standard operator $2$-norm,
and $\|A\|_F$ denotes its standard Frobenius norm.
%
%For a symmetric matrix $X$, $X\succeq 0$ (resp., $X\succ 0$) means
%$X$ is positive semidefinite (resp. positive definite).
%
%For a general matrix $X$, $\sig_{max}(X)$ and $\sig_{min}(X)$
%denote the maximum and minimum singular values of $X$ respectively.
%
For a complex vector $u$, $\| u \| := \sqrt{u^*u}$
denotes the standard Euclidean norm,
and $(u)_i$ denotes its $i$th entry.
%%
%For a tensor $\mc{F} \in \mt{T}^m(\cpx^{n+1})$, its norm $\| \mc{F} \|$
%is defined as in \reff{tensor:norm}.
%%
%We define its max-norm as
%\be \label{maxnorm:F}
%\| \mc{F} \|_{max} := \max_{0 \leq i_1,\ldots, i_m \leq n  }
%\, | \mc{F}_{ i_1,\ldots, i_m } |.
%\ee
%
For two square matrices $X,Y$ of the same dimension,
denote their commutator
\be \label{df:[X,Y]}
 [X, Y] := XY-YX.
\ee
The Hilbert-Schmidt norm of a tensor $\mF \in \mt{T}^m(\cpx^n)$
is defined as:
\be \label{tensor:norm}
\|\mF\|  :=
\Big( \sum_{ 1\leq i_1 ,  \ldots,   i_m \leq  n}
| \mc{F}_{i_1 \ldots i_m} |^2 \Big)^{1/2}.
\ee

Symmetric tensors can be equivalently indexed by monomial powers.
For convenience, let
$\on \, := \, n-1.$
For $\af := (\af_1,\ldots,\af_{\on}) \in \N^{\on}$
and $x := (x_1,\ldots, x_{\on})$, denote
\[
|\af| := \af_1 + \cdots + \af_{\on},
\quad
x^\af := x_1^{\af_1} \cdots x_{\on}^{\af_{\on}}.
\]
Denote
$
\N_m^{\on} := \{ \af  \in \N^{\on} : |\af|   \leq m \}.
$
Let $\cpx[x]:=\cpx[x_1,\ldots,x_{\on}]$ be the ring of polynomials in $x$,
with complex coefficients.
Each $\mc{F} \in \mt{S}^m(\cpx^{n})$ is indexed by an integer tuple
$(i_1,\ldots, i_m)$, with $1\leq i_1, \ldots, i_m \leq n$.
We can equivalently
index $\mF$ by $\af \in \N_m^{\on}$ as (let $x_0:=1$)
\be \label{idx:F:af=i1toim}
\mc{F}_{\af} := \mc{F}_{i_1, \ldots, i_m}
\quad \mbox{ whenever} \quad
x^\af = x_{i_1-1} \cdots x_{i_m-1}.
\ee
Throughout the paper, we mostly use the above monomial power indexing.

\subsection{Catalecticant matrices}
\label{ssc:Cat}

The Catalecticant matrix\footnote{
There are several Catalecticant matrices for $\mF$,
as in the reference \cite{IaKa99}.
In this paper, we only use the most square one.}
\cite{IaKa99} of a tensor $\mF \in \mt{S}^m(\cpx^n)$
is defined as
\be \label{df:CatMat:F}
\mbox{Cat}(\mF) := (\mF_{\af+\bt})_{|\af|\leq m_1, |\bt|\leq m_2},
\ee
where $m_1 = \lfloor \frac{m}{2} \rfloor$, $m_2 = \lceil \frac{m}{2} \rceil$
and $\mF$ is indexed as in \reff{idx:F:af=i1toim}.
The matrix $\mbox{Cat}(\mF)$ is also called
the symmetric flattening of $\mF$ in some references.
For a rank $r$, let $\sig_r$ be the closure
in the Zariski topology \cite{CLO07} of the set of tensors
$(u_1)^{\otimes m}+\cdots+(u_r)^{\otimes m}$,
with $u_1, \ldots, u_r \in \cpx^n$.
The set $\sig_r$ is an irreducible variety~\cite[Chapter~5]{Land12}.
It is also called the $r$th order secant variety of the degree $m$
Veronese embedding of $\cpx^n$ in the literature.
The {\it symmetric border rank} of $\mF$ is then defined as
\be \label{df:bordrank}
\rank_{SB}(\mc{F}) = \min \left\{r: \, \mF \in \sig_r \right\}.
\ee
It always holds that \cite[Lemma~2.1]{GPSTD}
\be  \label{rk:Cat<=B<=S}
\rank\,\mbox{Cat}(\mF) \leq  \rank_{SB}(\mF) \leq  \rank_S(\mF).
\ee
A property {\tt P} is said to be {\it generically} true on
$\sig_r$ if it is true in a nonempty Zariski open subset $T$ of $\sig_r$
\cite[Remark~2.1]{OedOtt13}.
For such a property {\tt P}, we say that $w$
is a {\it generic} point for {\tt P} if $w \in T$.
Interestingly, the inequalities in \reff{rk:Cat<=B<=S} become equalities
if $\mF$ is a generic point of $\sig_r$ and $r$ does not exceed
the smaller size of  $\mbox{Cat}(\mF)$.
This fact was also noted in Iarrobino and Kanev
\cite[page XVIII]{IaKa99}. For convenience of reference,
we give a straightforward proof for this fact.

\begin{lem} \label{pro:rk:Cat=S}
Let
$s = \min\{ \binom{n+m_1-1}{m_1}, \binom{n+m_2-1}{m_2} \}$,
the smaller size of $\mbox{Cat}(\mF)$.
For all $r \leq s$,
if $\mF$ is a generic point of $\sig_r$, then
\be \label{r<s:rkS=Cat}
\rank\,\mbox{Cat}(\mF) = \rank_{SB}(\mF) = \rank_S(\mF) = r.
\ee
\end{lem}
\begin{proof}
Let $\phi_1,\ldots, \phi_k$ be the $r\times r$ minors of the matrix
\[
\mbox{Cat} \Big( (x^1)^{\otimes m}+\cdots +(x^r)^{\otimes m} \Big)
\in \cpx^{ \binom{n+m_1-1}{m_1} \times \binom{n+m_2-1}{m_2} }.
\]
Denote $x:= (x^1,\ldots,x^r)$, with each $x^i \in \cpx^n$, then
$\phi_1,\ldots, \phi_k$ are homogeneous polynomials in $x$.
Let
\[
Z= \{x: \phi_1(x) =\cdots =\phi_k(x) = 0 \}.
\]
Let $Y:=(\cpx^n)^r\backslash Z$, a Zariski open set in $(\cpx^n)^r$.
Define the mapping:
\[
\pi:\, Y \to \sig_r, \quad (x^1,\ldots,x^r) \mapsto
(x^1)^{\otimes m}+\cdots +(x^r)^{\otimes m}.
\]
The image $\pi(Y)$ is dense in $\sig_r$.
So, $\pi(Y)$ contains a Zariski open subset, say, $\mathscr{Y}$, of $\sig_r$
\cite[Theorem~6,\S5,Chap.I]{Sha:BAG1}.
For each $\mF \in \mathscr{Y}$, we have
$\mF=(u_1)^{\otimes m}+\cdots +(u_r)^{\otimes m}$
for a tuple $u = (u_1,\ldots,u_r) \not\in Z$.
At least one of $\phi_1(u),\ldots,\phi_k(u)$ is nonzero,
so $\rank\,\mbox{Cat}(\mF) \geq r$.
By \reff{rk:Cat<=B<=S}, we have
$\rank\,\mbox{Cat}(\mF) \leq \rank_S(\mF) \leq r$,
which implies that \reff{r<s:rkS=Cat} is true.
\end{proof}

Lemma~\ref{pro:rk:Cat=S} implies that
if $r \leq s$, then for generic $\mF \in \sig_r$
we have $\rank_S(\mF)= \rank \, \mbox{Cat}(\mF) = r$.
So, in practice, $\rank_S(\mF)$ and $\rank_{SB}(\mF)$ can be estimated by
$\rank \, \mbox{Cat}(\mF)$ when they are smaller than or equal to $s$.
However, for generic $\mF \in \mt{S}^m( \cpx^n )$
such that $\rank \, \mbox{Cat}(\mF)=r$, we cannot conclude
$\mF \in \sig_r$.

\subsection{Generating polynomials}
\label{ssc:gp:sym}

This subsection mostly reviews the results in \cite{GPSTD}.
Let $\cpx[x]_m$ be the subset of polynomials
in $\cpx[x]$ whose degrees $\leq m$.
For $p = \sum_{\af\in\N_m^{\on} } p_\af x^\af \in \cpx[x]_m$
and $\mc{F} \in \mt{S}^m(\cpx^{n})$, define the product
\be \label{op:scrL:F}
\langle p, \mc{F} \rangle  := \sum_{\af\in\N_m^{\on} } p_\af \mc{F}_\af,
\ee
where $\mF$ is indexed as in \reff{idx:F:af=i1toim}
and each $p_\af$ is a coefficient.
Let $\deg(p)$ denote the total degree of a polynomial $p$.
As defined in \cite{GPSTD}, a polynomial $g \in \cpx[x]_m$
is called a {\em generating polynomial} for $\mF$ if
\be \label{df:GRpq}
\langle g \, x^\bt, \mc{F} \rangle  = 0 \quad
\forall\, \bt \in \N_m^{\bn}: \,
\deg(g) +  |\bt| \leq m.
\ee
The notion of generating polynomials is equivalent to
apolarity for symmetric tensors upon homogenization
\cite[Proposition~2.2]{GPSTD}.
Generating polynomials are useful for
computing symmetric tensor decompositions.
We consider generating polynomials for rank-$r$ symmetric tensors.
Let
\be \label{monls:grlex}
\mathbb{B}_0 := \Big\{
\underbrace{ 1, \, x_1, \, \ldots, \, x_{\bn}, \, x_1^2, \, x_1x_2,
\, \ldots }_{ \, \mbox{ the first $r$ } \, }
\Big\},
\ee
the set of first $r$ monomials
in the graded lexicographic ordering. Let
\be \label{mscrB12}
\mathbb{B}_1 := \big( \mathbb{B}_0 \cup x_1 \mathbb{B}_0
\cup \cdots \cup x_{\bn} \mathbb{B}_0)
\backslash \mathbb{B}_0.
\ee
The set $\mathbb{B}_1$ consists of the next $k$ monomials
in the graded lexicographic order,
where $k \leq rn$ but for which no closed expression is known.
For convenience, by writing $\bt \in \mathbb{B}_0$ (resp., $\af \in \mathbb{B}_1)$,
we mean that $x^\bt \in \mathbb{B}_0$ (resp., $x^\af \in \mathbb{B}_1)$.
Let $\cpx^{ \mathbb{B}_0 \times \mathbb{B}_1 }$ be the space of all complex matrices
indexed by $ (\bt, \af)  \in \mathbb{B}_0 \times \mathbb{B}_1$.
For $\af \in \mathbb{B}_1$ and $G \in \cpx^{ \mathbb{B}_0 \times \mathbb{B}_1 }$,
denote the polynomial in $x$
\be \label{vphi:W-af}
\varphi[G, \af](x) := \sum_{ \bt  \in \mathbb{B}_0 }
G(\bt,\af)  x^\bt - x^\af.
\ee
If all $\varphi[G, \af]$ ($\af \in \mathbb{B}_1$)
are generating polynomials,
then $G$ is called a {\it generating matrix} for $\mF$.
This requires that $\mF$ satisfies the linear relations
\cite[Prop.~3.5]{GPSTD}
\[
F_{\af + \gm } = \sum_{ \bt  \in \mathbb{B}_0 }  G(\bt,\af)  F_{\bt + \gm},
\]
for all $\gm \in \N^{\bn}$ with $|\gm| + |\af| \leq m$.
The above implies that the entire tensor $\mF$ can be determined
by the matrix $G$ and its first $r$ entries $\mF_\bt$
($\bt \in \bB_0$). This observation motivates the notion
of generating polynomials \cite[\S1.3]{GPSTD}. Denote
\be \label{vphi[W](x)}
\varphi[G](x) : = \big( \varphi[G, \af](x)\big)_{ \af \in \mathbb{B}_1 }.
\ee

The relations between generating polynomials and symmetric tensor decompositions
can be summarized as follows. Suppose $\mF$ has the decomposition
\be  \label{dcpA:u^m}
\mF = (u_1)^{\otimes m} + \cdots +  (u_r)^{\otimes m},
\ee
with $u_1,\ldots, u_r \in \cpx^n$.
Let $(u_i)_j$ be the $j$th entry of $u_i$.
If each $(u_i)_1 \ne 0$, let
\be \label{dhmg:ui:to:vi}
v_i = \big( (u_i)_2, \ldots, (u_i)_{\bn} \big)/ (u_i)_1, \quad
\lmd_i = \big( (u_i)_1 \big)^m.
\ee
Then, \reff{dcpA:u^m} is equivalent to the decomposition
\be \label{dcpA:lmd*v}
\mF = \lmd_1 \bbm 1 \\ v_1 \ebm^{\otimes m} + \cdots +
\lmd_r \bbm 1 \\ v_r \ebm^{\otimes m}.
\ee
The major part in computing \reff{dcpA:lmd*v}
is to determine $v_1,\ldots, v_r$. Once they are known,
the coefficients $\lmd_1, \ldots, \lmd_r$ can be
determined by solving linear equations.

\begin{theorem}  \label{thm:dstcV=>lmd}
(\cite{GPSTD})
Let $\mF \in \mt{S}^m(\cpx^n)$ be a symmetric tensor.
\bit

\item [(i)] If $G$ is a generating matrix of $\mF$ and $v_1,\ldots,v_r$
are distinct zeros of $\varphi[G](x)$,
then there exist scalars $\lmd_1,\ldots, \lmd_r$
satisfying \reff{dcpA:lmd*v}.

\item [(ii)] In \reff{dcpA:lmd*v}, if
$\det \big( [v_1]_{\mathbb{B}_0} \quad \cdots \quad
[v_r]_{\mathbb{B}_0}  \big) \ne 0$,
then there exists a unique generating matrix $G$ of $\mF$ such that
$v_1,\ldots,v_r$ are distinct zeros of $\varphi[G](x)$.
(The $[v_j]_{\mathbb{B}_0}$ denotes the vector
of monomials in $\bB_0$ evaluated at $v_j$.)
\eit
\end{theorem}

One wonders when $\varphi[G]$ has $r$ common zeros (counting multiplicities).
This question can be answered by using companion matrices.
For each $i=1,\ldots, \on$, the companion matrix for $\varphi[G](x)$
with respect to $x_i$ is $M_{x_i}(G)$, which is indexed by
$(\mu,\nu) \in \mathbb{B}_0 \times \mathbb{B}_0$
and is given as
\be \label{df:Mxi(W)}
\Big( M_{x_i}(G) \Big)_{\mu, \nu} =
\bca
1  &  \text{ if } x_i \cdot x^\nu \in \mathbb{B}_0 \text{ and }  \mu = \nu + e_i, \\
0  &  \text{ if } x_i \cdot x^\nu \in \mathbb{B}_0 \text{ and } \mu \ne \nu + e_i, \\
G(\mu, \nu+e_i)  &  \text{ if } x_i \cdot x^\nu \in \mathbb{B}_1.
\eca
\ee
Then, $\varphi[G]$ has $r$ common zeros (counting multiplicities)
if and only if the companion matrices
$M_{x_1}(G)$, $\ldots$, $M_{x_{\on}}(G)$ commute
\cite[Prop.~2.4]{GPSTD}, i.e.,
\be \label{MiMj=MjMi}
[M_{x_i}(G), \, M_{x_j}(G)] \, = \, 0 \quad
(1 \leq i < j \leq \on).
\ee

\subsection{Other related work on Waring decompositions}

There exist other methods for computing Waring decompositions,
which are related to apolarity and catalecticant matrices.
For binary symmetric tensors,
Sylvester's algorithm can be used to compute Waring decompositions.
For higher dimensional tensors,
the catalecticant matrix method is often used,
which often assumes the tensor rank is small \cite{IaKa99}.
Oeding and Ottaviani~\cite{OedOtt13}
proposed tensor eigenvector methods, which used
Koszul flattening and vector bundles.
We also refer to \cite{BalBer12,BerGimIda11,Kolda15}
for related work on symmetric tensor decompositions.

Brachat et al.~\cite{BCMT10} generalized Sylvester's algorithm
to higher dimensional tensors,
using properties of Hankel (and truncated Hankel) operators.
The method is to extend a given tensor $\mF \in \mt{S}^m(\cpx^n)$
to a higher order one
$\widetilde{\mF} \in \mt{S}^k(\cpx^n)$, with $k \geq m$.
The new tensor entries are treated as unknowns.
It requires to compute new entries of $\widetilde{\mF}$
such that the ideal, corresponding to the null space of
the catalecticant matrix $\mbox{Cat}(\widetilde{\mF})$,
is zero-dimensional and radical.
Mathematically, this is equivalent to solving a set of equations,
which we refer to Algorithm~7.1 of \cite{BCMT10}.
%%%%%%%%%%%%%%%%%%%%%
\iffalse
Theoretically, this method can find Waring decompositions for all tensors.
However, it is often very expensive to be implemented computationally.
One reason is that $\widetilde{\mF}$
has a large number of unknown entries,
which need to determined by solving some rational equations.
\fi
%%%%%%%%%%%%%%%%%%%%%%%%

The methods in \cite{BCMT10} and \cite{GPSTD} are both
related to apolarity. They have similarities,
but they are also computationally different.
To see this, we consider the tensor
$\mF \in \mt{S}^3(\cpx^3)$ in Example~5.1 of \cite{GPSTD},
whose entries $\mF_{00}$, $\mF_{10}$, $\mF_{01}$,
$\mF_{20}$, $\mF_{11}$, $\mF_{02}$,
$\mF_{30}$, $\mF_{21}$, $\mF_{12}$, $\mF_{03}$
(labelled by monomial powers as in \reff{idx:F:af=i1toim}) are respectively
$-8,2,15,-7,17$,$7,17,4,3,18$.
%\[
%\baray{rrrrrrrrrr}
%\mF_{00} & \mF_{10} & \mF_{01} &
%\mF_{20} & \mF_{11} & \mF_{02} &
%\mF_{30} & \mF_{21} & \mF_{12} & \mF_{03} \\
%-8 & 2 & 15 & -7 & 17 & 7 & 17 & 4 & 3 & 18
%\earay
%\]
It has rank $r=4$.
To apply the method in \cite{BCMT10} with $r=4$,
one needs to extend $\mF$ to $\widetilde{\mF} \in \mt{S}^4(\cpx^3)$
with new entries $\mF_{ij}$\,($i+j > 3$).
Their values need to be computed by solving
the equation $C_1C_2-C_2C_1=0$ where
\[
C_1 =
\bbm
2 & -7 & 17  & 17 \\
-7 & 17 & 4  & \widetilde{\mF}_{40}  \\
17 & 4 & 3  & \widetilde{\mF}_{31} \\
17 & \widetilde{\mF}_{40} & \widetilde{\mF}_{31}  & \widetilde{\mF}_{50} \\
\ebm
%\bbm
%\mF_{10} & \mF_{20} & \mF_{11}  & \mF_{30} \\
%\mF_{20} & \mF_{30} & \mF_{21}  & \mF_{40}  \\
%\mF_{11} & \mF_{21} & \mF_{12}  & \mF_{31} \\
%\mF_{30} & \mF_{40} & \mF_{31}  & \mF_{50} \\
%\ebm
\bbm
-8  & 2 & 15  & -7 \\
2 & -7 & 17  & 17  \\
15 & 17 & 7  & 4\\
-7 & 17 & 4  & \widetilde{\mF}_{40} \\
\ebm^{-1},
%\bbm
%\mF_{00} & \mF_{10} & \mF_{01}  & \mF_{20} \\
%\mF_{10} & \mF_{20} & \mF_{11}  & \mF_{30}  \\
%\mF_{01} & \mF_{11} & \mF_{02}  & \mF_{21} \\
%\mF_{20} & \mF_{30} & \mF_{21}  & \mF_{40} \\
%\ebm
\]
\[
C_2 =
\bbm
15 & 17 & 7  & 4 \\
17 & 4 & 3  & \widetilde{\mF}_{31}  \\
7 & 3 & 18  & \widetilde{\mF}_{22} \\
4 & \widetilde{\mF}_{31} & \widetilde{\mF}_{22}  & \widetilde{\mF}_{41} \\
\ebm
%\bbm
%\mF_{01} & \mF_{11} & \mF_{02}  & \mF_{21} \\
%\mF_{11} & \mF_{21} & \mF_{12}  & \mF_{31}  \\
%\mF_{02} & \mF_{12} & \mF_{03}  & \mF_{22} \\
%\mF_{21} & \mF_{31} & \mF_{22}  & \mF_{41} \\
%\ebm
\bbm
-8  & 2 & 15  & -7 \\
2 & -7 & 17  & 17  \\
15 & 17 & 7  & 4\\
-7 & 17 & 4  & \widetilde{\mF}_{40} \\
\ebm^{-1}.
%\bbm
%\mF_{00} & \mF_{10} & \mF_{01}  & \mF_{20} \\
%\mF_{10} & \mF_{20} & \mF_{11}  & \mF_{30}  \\
%\mF_{01} & \mF_{11} & \mF_{02}  & \mF_{21} \\
%\mF_{20} & \mF_{30} & \mF_{21}  & \mF_{40} \\
%\ebm
\]
To apply the method in \cite{GPSTD} with $r=4$,
one needs to determine the generating polynomials,
which can be parameterized as follows:
%%%%%%%%%%%%%%
\iffalse

clear all
syms a1 a2 a3 b1 b2 b3  c1 c2 c3;
tdeqn =[...
-8 == a1^3+b1^3+c1^3, ...
2 == a1^2*a2+b1^2*b2+c1^2*c2, ...
15 == a1^2*a3+b1^2*b3+c1^2*c3, ...
-7 == a1*a2^2+b1*b2^2+c1*c2^2, ...
17 == a1*a2*a3+b1*b2*b3+c1*c2*c3, ...
7 == a1*a3^2+b1*b3^2+c1*c3^2, ...
17 == a2^3+b2^3+c2^3, ...
4 == a2^2*a3+b2^2*b3+c2^2*c3, ...
3 == a2*a3^2+b2*b3^2+c2*c3^2, ...
18 == a3^3+b3^3+c3^3];
Sol = solve( tdeqn );
double(Sol.a1),

clear all

syms omg_1 omg_2 omg_3 omg_4 omg_5 omg_6 omg_7  omg_8;

AF11 = [-8 2  15  -7; 2  -7   17  17; 15  17  7  4 ];
bF11 = [17 ; 4 ; 3];
g11 = inv([ AF11 ; 1 0 0 0 ])*([bF11; omg_1]),

AF02 = [-8 2  15  -7; 2  -7   17  17; 15  17  7  4 ];
bF02 = [7; 3; 18];
g02 = inv([ AF02 ; 1 0 0 0 ])*([bF02; omg_2]),

AF30 = [-8  2  15  -7];
bF30 = [17];
g30 = inv([ AF30 ; 1 0 0 0; 0 1 0 0; 0 0 1 0 ])*([bF30; omg_3; omg_4 ; omg_5]),

AF21 = [-8  2  15  -7];
bF21 = [4];
g21 = inv([ AF21 ; 1 0 0 0; 0 1 0 0; 0 0 1 0 ])*([bF21; omg_6; omg_7; omg_8]),

\fi
%%%%%%%%%%%%%%%%%%%%
\[
\omega_1 + \text{\tiny $\frac{ -5800 \omega_1 - 181}{7019}$ } x_1 +
\text{\tiny  $\frac{ 2057 \omega_1  + 5942}{7019}$ }  x_2
+  \text{\tiny $\frac{ -5271 \omega_1  - 4365}{7019}$ } x_1^2  - x_1x_2,
\]
\[
\omega_2 +  \text{\tiny $\frac{6048 - 5800 \omega_2 }{7019}$ } x_1 +
\text{\tiny $\frac{ 2057 \omega_2  + 2870 }{7019}$  } x_2 +
\text{\tiny $\frac{ 859 - 5271 \omega_2}{7019}$  } x_1^2  - x_2^2,
\]
\[
\omega_3 +   \omega_4 x_1 +  \omega_5  x_2 +
\text{\tiny $\frac{ 2 \omega_4 - 8 \omega_3 + 15 \omega_5 - 17}{7}$ } x_1^2  - x_1^3,
\]
\[
\omega_6 +   \omega_7 x_1 +  \omega_8  x_2 +
\text{\tiny $\frac{ 2 \omega_7 - 8 \omega_6 + 15 \omega_8 - 4}{7}$  } x_1^2  - x_1^2x_2.
\]
The parameters $\omega_1, \omega_2, \ldots, \omega_8$
can be determined by solving the equation $M_1M_2-M_2M_1=0$, where
\[
M_1 =
\bbm
0  &  0  &  \omega_1   &   \omega_3 \\
1  &  0  &  \text{\tiny $\frac{ -5800 \omega_1 - 181}{7019}$  }  &  \omega_4  \\
0  &  0  &  \text{\tiny $\frac{ 2057 \omega_1  + 5942}{7019}$  }  &  \omega_5  \\
0  &  1  &  \text{\tiny $\frac{ -5271 \omega_1  - 4365}{7019}$  }
           & \text{\tiny $\frac{ 2 \omega_4 - 8 \omega_3 + 15 \omega_5 - 17}{7} $ }
\ebm,
\]
\[
M_2 =
\bbm
0 & \omega_1 &  \omega_2  & \omega_6 \\
0 & \text{\tiny $\frac{ -5800 \omega_1 - 181}{7019}$ }
            & \text{\tiny $\frac{6048 - 5800 \omega_2 }{7019}$ } & \omega_7 \\
1 & \text{\tiny $\frac{ 2057 \omega_1  + 5942}{7019}$  }
          & \text{\tiny $\frac{ 2057 \omega_2  + 2870 }{7019}$  }  & \omega_8 \\
0 & \text{\tiny $\frac{ -5271 \omega_1  - 4365}{7019}$ }
         & \text{\tiny $\frac{ 859 - 5271 \omega_2}{7019}$   }
        &  \text{\tiny $\frac{ 2 \omega_7 - 8 \omega_6 + 15 \omega_8 - 4}{7}$  }
\ebm.
\]
Two of the parameters $\omega_1,\ldots, \omega_8$ can be eliminated
by adding two generic linear equations.
We refer to \S4.1 of \cite{GPSTD} for more details about the above.

For generic tensors of certain ranks, the Waring decomposition is unique.
In applications, the uniqueness justifies that
the computed decomposition is what people wanted.
Galuppi and Mella \cite{GalMel} showed that
for a generic $\mF \in \mt{S}^m(\cpx^n)$,
the Waring decomposition is unique if and only if
$(n,m,r) = (2,2k-1,2k)$, $(4,3,5)$ or $(3,5,7)$.
When $\mF \in \mt{S}^m(\cpx^n)$ is a generic tensor of a subgeneric rank $r$
(i.e., $r$ is smaller than the value of the Alexander-Hirschowitz formula) and $m\geq 3$,
Chiantini, Ottaviani and Vannieuwenhoven~\cite{ChOtVan15}
showed that the Waring decomposition is unique,
with only three exceptions:
$(n,m,r) = (2,6,9)$, $(3,4,8)$ or $(5,3,9)$.
In all of these three exceptions,
there are exactly two Waring decompositions.

\section{Low rank approximations}
\label{sc:lra:sym}
\setcounter{equation}{0}

For a symmetric tensor that has rank $r$,
Theorem~\ref{thm:dstcV=>lmd} can be used to
compute its Waring decomposition, as in \cite{GPSTD}.
The same approach can be extended to compute low rank approximations.
Given $\mc{F} \in \mt{S}^m(\cpx^n)$,
we are looking for a good approximation
$\mc{X} \in \mt{S}^m(\cpx^n)$ of $\mF$
such that $\rank_S(\mc{X}) \leq r$.
Note that $\rank_S(\mc{X}) \leq r$ if and only if
there exist vectors $u_1, \ldots, u_r \in \cpx^n$ such that
\[
\mc{X} = (u_1)^{\otimes m} + \cdots + (u_r)^{\otimes m}.
\]
Finding the best rank-$r$ approximation is equivalent to
solving the nonlinear nonconvex optimization problem~\reff{nLS:lra:symF}.
%
%\be \label{apx:LS:u}
%\min_{ u_1, \ldots, u_r \in \cpx^n } \quad
%\left \|  (u_1)^{\otimes m} + \cdots
%+ (u_r)^{\otimes m} - \mF \right \|^2.
%\ee
%

Clearly, if $\rank_S(\mF)=r$, the best rank-$r$ approximation
is given by the rank decomposition of $\mF$.
When $\mF$ is close to the set of rank-$r$ tensors,
the best rank-$r$ approximation is given by the rank decomposition of
a rank-$r$ tensor that is close to $\mF$.
Decompositions of rank-$r$ tensors can be computed
by using generating polynomials as in~\cite{GPSTD}.
This motivates us to compute low rank approximations
by using generating polynomials. Let
\[ \bn := n-1. \]
For $u \in \cpx^n $ with $(u)_1 \ne 0$, we can write it as
\[
u = \lmd^{1/m} (1, v), \quad \lmd \in \cpx, \, v \in \cpx^{\on}.
\]
Then, $u^{\otimes m}  = \lmd  (1,v)^{\otimes m}$, and
\reff{nLS:lra:symF} can be reformulated as
\be \label{apx:LS:lmd*v}
\min_{
\substack{v_1, \ldots, v_r \in \cpx^{\on} \\ \lmd_1, \ldots, \lmd_r \in \cpx  }
} \quad
\| \lmd_1 (1, v_1)^{\otimes m} + \cdots + \lmd_r (1, v_r)^{\otimes m}
- \mF \|^2.
\ee
We propose to solve \reff{apx:LS:lmd*v} in three major stages:
\bit

\item [1)] Find a matrix $G$ such that the polynomials
$\varphi[G,\af]$ as in \reff{vphi:W-af}
approximate generating polynomials for $\mc{F}$
as much as possible.

\item [2)] Compute vectors $v_1,\ldots,v_r$
that are approximately common zeros
of the generating polynomials $\varphi[G,\af](x)$.

\item [3)] Determine $\lmd_1,\ldots,\lmd_r$
and construct low rank approximations.

\eit

\subsection{Estimate generating polynomials}
\label{ssc:aprx:gp}

For a given rank $r$, let $\mathbb{B}_0,\mathbb{B}_1$
be the monomial sets as in \reff{monls:grlex}-\reff{mscrB12}.
%
%By writing $\af \in \mathbb{B}_1$ we mean that $x^\af \in \mathbb{B}_1$.
%We index the tensor $\mF$ by monomial powers as in \reff{idx:F:af=i1toim}.
%
For a matrix $G \in \cpx^{\mathbb{B}_0 \times \mathbb{B}_1}$
and $\af \in \mathbb{B}_1$, the $\varphi[G,\af]$
as in \reff{vphi:W-af} is a generating polynomial for $\mc{F}$
if and only if
\be \label{<gmWaf:F>=0}
\sum_{\bt \in \mathbb{B}_0}  G(\bt, \af) \mF_{\bt + \gm} = \mF_{\af + \gm}
%\Big \langle x^\gm \varphi[G, \af], \mc{F} \Big\rangle = 0
\quad \big( \forall\, \gm \in \N_{m-|\af|}^{\on} \big),
\ee
which is implied by \reff{df:GRpq}.
A matrix $G$ satisfying \reff{<gmWaf:F>=0} for all $\af \in \mathbb{B}_1$
generally exists if $\rank_S(\mF) \leq r$,
but it might not if $\rank_S(\mF) > r$.
However, we can always find the linear least squares solution
of \reff{<gmWaf:F>=0}.
%%%%%%%%%%%%%%%%%%%%%%%%%%%%%%%%%%%%%%%%
\iffalse

\be \label{ls:rc:Waf}
\min_{ G \in \cpx^{\mathbb{B}_0 \times \mathbb{B}_1} } \quad
\sum_{ \af \in \mathbb{B}_1 }  \sum_{ \gm \in  \N_{m-|\af|}^{\on} }
\Big| \big
\langle x^\gm \varphi[G,\af], \mc{F}
\big\rangle \Big|^2.
\ee

\fi
%%%%%%%%%%%%%%%%%%%%%%%%%%%%%%%%%%%%%%%%%%%%%
Indeed, for each $\af \in \bB_1$,
let the matrix $A[\mc{F},\af]$ and the vector $b[\mc{F},\af]$ be such that
\be \label{df:Ab[F,af]}
\left\{ \baray{lcl}
A[\mc{F},\af]_{\gm , \bt} &=& \mc{F}_{\bt+\gm}, \quad
\forall \,  (\gm, \bt) \in \N_{m-|\af|}^{\on} \times \mathbb{B}_0, \\
b[\mc{F},\af]_{\gm} &=& \mc{F}_{\af+\gm}, \quad
\forall \,  \gm \in \N_{m-|\af|}^{\on}.
\earay\right.
\ee
The dimension of $A[\mc{F},\af]$ is
$\binom{\on+m-|\af|}{m-|\af|} \times r$,
and the length of $b[\mc{F},\af]$ is
$\binom{\on+m-|\af|}{m-|\af|}$.
Let $G(:,\af)$ denote the $\af$th column of $G$,
then \reff{<gmWaf:F>=0} is equivalent to
\[
A[\mc{F},\af] G(:, \af) = b[\mc{F},\af].
\]
Consider the linear least squares problem
\be \label{ls:Aw=b}
\min_{ G \in \cpx^{\mathbb{B}_0 \times \mathbb{B}_1} } \quad
\sum_{ \af \in \mathbb{B}_1 }
\Big\| A[\mc{F},\af] \, G(:,\af) - b[\mc{F},\af] \Big\|^2.
\ee
Each summand in \reff{ls:Aw=b} involves a different column of $G$.
So, \reff{ls:Aw=b} can be decoupled into a set of
smaller linear least squares problems
(for each $\af \in \mathbb{B}_1$):
\be \label{ls:Aw=b:af}
\min_{ g \in \cpx^{\mathbb{B}_0 } } \quad
\Big\| A[\mc{F},\af] \, g - b[\mc{F},\af] \Big\|^2.
\ee

When $r \leq \binom{\on+m-|\af|}{m-|\af|}$, for generic $\mF$,
it is expected that $A[\mc{F},\af]$ has linear independent columns.
For such a case, \reff{ls:Aw=b:af} has a unique least squares solution.
When $r > \binom{\on+m-|\af|}{m-|\af|}$,
the least squares solution to \reff{ls:Aw=b:af} is not unique,
but it can be linearly parameterized.
Let $N_\af$ be a basis matrix for the null space of $A[\mc{F},\af]$,
if it exists. The optimal solution of \reff{ls:Aw=b:af}
can be parameterized as
\be \label{para:g(omg:af)}
g_{\af}(\omega_\af)  := g_{\af}^{ls}  + N_\af \omega_\af,
\ee
with $g_{\af}^{ls}$ the minimum norm solution to
\reff{ls:Aw=b:af} and $\omega_\af$ a parameter.
When $A[\mc{F},\af]$ has full column rank,
$N_\af$ and $\omega_\af$ do not exist. Let
\be \label{omg:all}
\omega = [\omega_\af]_{\af \in \mathbb{B}_1 }
\ee
be the vector of all parameters.
The optimal solution to \reff{ls:Aw=b} is the matrix
\be \label{G(omg):par}
G (\omega) \, := \, [g_{\af}(\omega_\af)]_{ \af \in \mathbb{B}_1 }
\in \cpx^{ \mathbb{B}_0 \times   \mathbb{B}_1 }.
\ee
%
%The $\af$th column of $G(\omega) $ is $g_{\af}(\omega_\af)$.
%
Our goal is to find $\omega$ such that
$\varphi[G(\omega)](x)$, defined in \reff{vphi[W](x)},
has (or is close to have) $r$ common zeros.
By Proposition~2.4 of \cite{GPSTD},
this requires that the companion matrices in \reff{df:Mxi(W)}
\[
M_{x_1}( G (\omega) ), \ldots,  M_{x_{\on}}( G (\omega) )
\]
are as commutative as possible.
So, we propose to compute $\omega$ by solving the optimization problem
\be \label{nls:omg:cmmu}
\min_{ \omega }  \quad \sum_{ 1 \leq i < j \leq \on }
\Big \| \Big [M_{x_i} ( G (\omega) ), \,\,
M_{x_j}( G (\omega) ) \Big]  \Big\|_F^2.
\ee
Let $\hat{\omega}$ be an optimizer of \reff{nls:omg:cmmu} and
\be \label{Gls}
G^{ls} := G ( \hat{\omega} ).
\ee

\subsection{Approximate common zeros}
\label{ssc:acom:zr}

%
%As shown in \cite{GPSTD}, if $\mF$ is a generic rank-$r$
%symmetric tensor, then the polynomial system
%\be \label{vpiWLS(x)=0}
%\varphi[G^{ls}, \af](x) = 0 \, \, (\af \in \mathbb{B}_1)
%\ee
%has $r$ distinct common solutions.
%
%When $\mF$ is close to $\sig_r^{m,n}$,
%the set of tensors in $\mt{S}^m(\cpx^n)$ whose symmetric
%border ranks $\leq r$ (cf.~\S\ref{sbsc:symtsr}),
%

Recall that $\varphi[G](x)$ is defined as in \reff{vphi[W](x)}.
By Proposition~2.4 of \cite{GPSTD}, the polynomial system
\be \label{vpiWLS(x)=0}
\varphi[G^{ls}](x) = 0
\ee
has $r$ complex solutions (counting multiplicities) if and only if
\[
[M_{x_i}(G^{ls}),    M_{x_j}(G^{ls}) ] = 0
\quad (0 \leq i,j \leq \bn).
\]
If no solutions are repeated, they can be computed as follows. Denote
\be \label{M(xiWls)}
L(\xi) := \xi_1 M_{x_1}(G^{ls}) +
\cdots +  \xi_{\on} M_{x_{\on}}(G^{ls}),
\ee
where $\xi :=(\xi_1, \ldots, \xi_{\on}) \in \cpx^{\bn}$ is generically chosen.
Compute the Schur decomposition
\[
Q^* L( \xi ) Q \, = \, T,
\]
where $Q=[q_1 \, \ldots \, q_r ]$ is unitary and
$T$ is upper triangular. Then, it is well-known that the vectors
\[
\left(q_i^* M_{x_1}(G^{ls})q_i, \, \, \ldots, \, \,
q_i^* M_{x_{\on}}(G^{ls}) q_i \right) \quad
(i=1,\ldots, r)
\]
are the solutions to \reff{vpiWLS(x)=0}.
We refer to \cite{CGT97} for the details.

When $\mF$ is not a rank-$r$ tensor, the equation
\reff{vpiWLS(x)=0} typically does not have a solution.
However, when $\mF$ is close to $\sig_r$,
\reff{vpiWLS(x)=0} is expected to have $r$ common approximate solutions.
In such a case, what is the best choice for $\xi$?
We want to choose $\xi$ such that $L(\xi)$
maximally commutes with each $M_{x_i}(G^{ls})$.
This leads to the quadratic optimization problem
\be \label{best:xi:Mxi}
\underset{\xi \in \cpx^{\on}, \| \xi \|_2 = 1 }{\min}
\quad \sum \limits_{i=1}^{\on}
\Big\|  \Big[ M_{x_i}(G^{ls}) , \,\, L(\xi) \Big] \Big \|_F^2 .
\ee
It can be solved as an eigenvalue problem. Write, for all $i$,
\[
\big[ M_{x_i}(G^{ls}), \,\, L(\xi) \big] =
\sum_{j=1}^{\on} \xi_j M^{i,j}.
\]
Note that $M^{i,j}$ is just the commutator $[M_{x_i}(G^{ls}), M_{x_j}(G^{ls})]$.
Let $vec(M^{i,j})$
denote the vector consisting of the columns of $M^{i,j}$ and
\[
V_i = \bbm vec(M^{i,1}) & \cdots & vec(M^{i,\on}) \ebm.
\]
Then,
\[
\sum_{i=1}^{\on}
\left \| \Big[ M_{x_i}(G^{ls}), \,\, L(\xi) \Big]  \right \|_F^2 =
\xi^* \Big( \sum\limits_{i=1}^{\on} V_i^* V_i \Big) \xi.
\]
An optimizer $\hat{\xi}$ of \reff{best:xi:Mxi}
is an eigenvector (normalized to have unit length),
associated to the smallest eigenvalue of the Hermitian matrix
\be \label{hatV}
\widehat{V} \, := \, V_1^* V_1 + \cdots + V_{\on}^* V_{\on}.
\ee
As before, we compute the Schur decomposition
\be  \label{Schur:MWls}
\hat{Q}^* L( \hat{\xi} ) \hat{Q} \, = \, \hat{T},
\ee
where $\hat{Q}=[\hat{q}_1 \, \ldots \, \hat{q}_r ]$ is unitary and
$\hat{T}$ is upper triangular.
For $i=1,\ldots,r$, let
\be \label{hatv:ls:aprx}
v_i^{ls} :=
\left(\hat{q}_i^* M_{x_1}(G^{ls})\hat{q}_i, \, \, \ldots, \, \,
\hat{q}_i^* M_{x_{\on}}(G^{ls}) \hat{q}_i \right).
\ee
The vectors $v_1^{ls}, \ldots, v_r^{ls}$ can be used as approximate
solutions to \reff{vpiWLS(x)=0}.

\subsection{Construction of low rank tensors}
\label{ssc:con:lrt}

Once $v_1^{ls}, \ldots, v_r^{ls}$ are computed,
we can construct a low rank approximation
by solving the linear least squares problem
\be  \label{LS:vls*lmd=F}
\min_{ \substack{ (\lmd_1,\ldots,\lmd_r) \in \cpx^r } } \quad
\| \lmd_1 (1, v_1^{ls}) ^{\otimes m} + \cdots +
\lmd_r (1, v_r^{ls})^{\otimes m} -\mc{F} \|^2.
\ee
Let $\lmd^{ls}:=(\lmd_1^{ls}, \ldots, \lmd_r^{ls})$
be an optimal solution to \reff{LS:vls*lmd=F}.
For $i=1,\ldots,r$, let
\[
u_i^{ls} := \sqrt[m]{\lmd_i^{ls}} (1, v_i^{ls}).
\]
Then, we get the low rank tensor
\be \label{Xls:sym}
\mc{X}^{gp} := (u_1^{ls})^{\otimes m}+ \cdots + (u_r^{ls})^{\otimes m}.
\ee

%%%%%%%%%%%%%%%%%%%%
%%%% the following is dropped
%%%%%%%%%%%%%%%%%%%%%%%%%%%%%%%%

\iffalse

For the basic case $r=1$, we can get an explicit formula for
the approximating tensor $\mc{X}^{gp}$.
For such case, $\mathbb{B}_0=\{1\}$, $\mathbb{B}_1= \{x_1, \ldots, x_n\}$,
and each $A[\mc{F}, e_i]$ is the vector
\[
a_i \, := \, (\mc{F}_\gm)_{ \gm \in \N_{m-1}^{\on} }.
\]
For $i=1,\ldots,\on$, let
\be \label{r=1:fml:taui}
\theta_i \, := \, a_i^*b[\mc{F}, e_i] / a_i^*a_i.
\ee
Then, $v^{ls} \, = \, (1, \theta_1, \ldots, \theta_{\on})$ and (let $\theta_0=1$)
\be \label{r=1:fml:lmdls}
\lmd^{ls} =
\Big( \sum_{ 1 \leq i_1, \ldots, i_m \leq n}
\mc{F}_{i_1,\ldots,i_m}  \theta_{i_1-1}^* \cdots \theta_{i_m-1}^*
\Big)/ \| (v^{ls})^{\otimes m} \|^2.
\ee
So, $\mc{X}^{gp} = \lmd^{ls} (v^{ls})^{\otimes m}$
can be given by a closed formula.

\fi
%%%%%%%%%%%%%%%%%%%%%%%%%%%%%%%%%%%%%%%%%%%%%%%%%%%

%
%\subsection{An algorithm for low rank approximations}
%\label{sc:lrap:alg}
%

When $\mF$ is of rank $r$, the tensor $\mc{X}^{gp}$
is equal to $\mF$. When $\mF$ is close to a rank-$r$ tensor,
$\mc{X}^{gp}$ is expected to be a good rank-$r$ approximation.
We will show in the next section that if
$\mF$ is sufficiently close to $\sig_r$,
then $\mc{X}^{gp}$ is a quasi-optimal rank-$r$ approximation, i.e.,
its error is at most a constant multiple of the optimal error.
Moreover, we can always improve it
by solving the optimization problem \reff{nLS:lra:symF}.
Generally, this can be done efficiently, because
$\mc{X}^{gp}$ is a good approximation.

\subsection{A numerical algorithm}

Combining the above, we propose the following algorithm
for computing low rank approximations for symmetric tensors.

\begin{alg} \label{alg:tsr:aprox}
(Low rank symmetric tensor approximations.) \\
For a given tensor $\mc{F} \in \mt{S}^m(\cpx^n)$ and a rank $r$,
do the following:
\bit

\item [Step 1] Solve the linear least squares problem~\reff{ls:Aw=b}
and express its optimal solution as in \reff{G(omg):par}.
If the parameter $\omega$ exists, use a suitable numerical optimization method
to solve \reff{nls:omg:cmmu} and get $G^{ls}$ as in \reff{Gls}.

\item [Step 2] Let $\hat{\xi}$ be a unit length eigenvector
of $\widehat{V}$ corresponding to its smallest eigenvalue.
Compute the Schur decomposition \reff{Schur:MWls}.

\item [Step 3] Compute $v_1^{ls},\ldots, v_r^{ls}$ as in
\reff{hatv:ls:aprx}, and solve \reff{LS:vls*lmd=F} for
$(\lmd_1^{ls}, \ldots, \lmd_r^{ls})$.

\item [Step 4] Construct the tensor $\mc{X}^{gp}$ as in \reff{Xls:sym}.

\item [Step 5] Use a suitable numerical optimization method
to solve \reff{nLS:lra:symF} for
an improved solution $(u_1^{opt},\ldots, u_r^{opt})$,
with the starting point $(u_1^{ls},\ldots, u_r^{ls})$.
Output the tensor
\be \label{Xgr:sym}
\mc{X}^{opt} := (u_1^{opt})^{\otimes m}+ \cdots + (u_r^{opt})^{\otimes m}.
\ee

\eit

\end{alg}

The complexity of Algorithm~\ref{alg:tsr:aprox}
can be estimated as follows.
The linear least squares problem~\reff{ls:Aw=b}
consists of $|\bB_1|$ subproblems \reff{ls:Aw=b:af}.
Solving \reff{ls:Aw=b:af} requires the storage of a
$\binom{\bn+m-|\af|}{m-|\af|} \times r$ matrix
and $O(n^{m-|\af|}r^2)$ floating point operations (flops).
There are $|\bB_1| = O(nr)$ such subproblems in total.
Computing $\hat{\xi}$ requires a unit eigenvector
corresponding to the minimum eigenvalue of $\widehat{V}$,
which requires the storage of a $n \times n$
Hermitian matrix and essentially takes $O(n^3)$ flops.
The computation in \reff{hatv:ls:aprx} takes $O(nr^3)$ flops.
The Schur decomposition \reff{Schur:MWls} requires
the storage of a $r\times r$ matrix and essentially $O(r^3)$ flops.
The linear least squares \reff{LS:vls*lmd=F} requires
the storage of a $n^m \times r$ matrix and $O(n^{m}r^2)$ flops.
When the parameter $\omega$ in \reff{omg:all} exists,
the nonlinear least squares problem~\reff{nls:omg:cmmu} needs to be solved.
After $\mc{X}^{gp}$ is obtained,
the nonlinear least squares problem~\reff{nLS:lra:symF}
needs to be solved. The major difficult parts of Algorithm~\ref{alg:tsr:aprox}
are to solve \reff{nls:omg:cmmu} in the Step~1 and \reff{nLS:lra:symF} in the Step~5.
The cost for solving them depends on the choice of starting points.
Generally, it is hard to estimate the complexity
for solving nonlinear least squares problems.
We refer to \cite{yyx11}
for a survey of methods for solving nonlinear least squares.
Classical algorithms for solving linear least squares,
symmetric eigenvalue problems and Schur decompositions
are numerically stable. We refer to \cite{Demmel}
for complexity and stability issues in numerical linear algebra.

\subsection{Some remarks}
\label{ssc:rks}

\quad

%\begin{remark} \label{est:rk:cat}
\bigskip \noindent
{\bf The choice of rank $r$} \,
In practice, the rank $r$ is often not known in advance.
Theoretically, it can be very hard to get the best value of $r$.
However, we can estimate it from the Catalecticant matrix $\mbox{Cat}(\mF)$.
If $\mF = \mc{X} + \mc{E}$, then
\[
\mbox{Cat}(\mF) = \mbox{Cat}(\mc{X}) + \mbox{Cat}(\mc{E}).
\]
If $\mc{X}$ is a generic point of $\sig_r$ and $r \leq s$
(the smaller size of $\mbox{Cat}(\mF)$),
then $\rank_S \,(\mc{X}) = \rank\,\mbox{Cat}(\mc{X}) =r$,
by Lemma~\ref{pro:rk:Cat=S}. Hence,
$\rank_S \,(\mc{X})$ can be estimated by $\rank \,\mbox{Cat}(\mc{X})$.
When $\mc{E}$ is small, $\rank\,\mbox{Cat}(\mc{X})$
can be estimated by the numerical rank of $\mbox{Cat}(\mF)$ as follows:
compute the singular values of $\mbox{Cat}(\mF)$, say,
$\eta_1 \geq \eta_2 \geq \cdots \geq \eta_s \geq 0$.
If $\eta_r$ is significantly bigger than $\eta_{r+1}$, then
such $r$ is a good estimate for $\rank_S (\mc{X})$.
The border rank can also be estimated in the same way.
Evaluating the rank of a matrix numerically is a classical problem
in numerical linear algebra. We refer to the book \cite{Demmel}.
%\end{remark}

\bigskip \noindent
{\bf The case of real tensors} \,
When $\mF$ is a real symmetric tensor, i.e., $\mF \in \mt{S}^m(\re^n)$,
Algorithm~\ref{alg:tsr:aprox} can still be applied.
However, the computed low rank tensor
$\mc{X}^{gp}$ (and hence $\mc{X}^{opt}$)
may not be real any more.
This is because, in the Step~2, the Schur decomposition~\reff{Schur:MWls}
is over the complex field, even if the matrix
$L(\hat{\xi})$ is real. The reason is that
the eigenvalues of a real matrix are often not real.
For instance, the tensors in
Examples~\ref{exmp-sin} and \ref{exmp-rootsum}
are real, but the low rank approximating tensors
produced by Algorithm~\ref{alg:tsr:aprox} are not real.

\bigskip \noindent
{\bf About uniqueness} \,
When the least squares \reff{ls:Aw=b} has a unique solution,
i.e., the parameter $\omega$ does not exist,
the tensor $\mc{X}^{gp}$ in the Step~4 is uniquely
determined by $\mF$. However, the tensor $\mc{X}^{opt}$
might not be unique. This is because the optimization problem~\reff{nLS:lra:symF}
might have more than one minimizer, or it does not have
a minimizer and $\mc{X}^{opt}$ is only approximately optimal,
which is then not unique.
When \reff{ls:Aw=b} does not have a unique least squares solution,
the tensor $\mc{X}^{gp}$ is not unique
if \reff{nls:omg:cmmu} does not have a unique minimizer.
Consequently, $\mc{X}^{opt}$ might also not be unique.
On the other hand, if both \reff{nls:omg:cmmu} and \reff{nLS:lra:symF}
have unique optimizers, the approximating tensors
$\mc{X}^{gp}$ and $\mc{X}^{opt}$ are unique.

\section{Approximation error analysis}
\label{sbsc:symerr}
\setcounter{equation}{0}

We analyze the approximation quality of
the low rank tensors $\mc{X}^{gp}, \mc{X}^{opt}$
produced by Algorithm~\ref{alg:tsr:aprox}. Suppose
\be \label{dcmp:Xbs}
\mc{X}^{bs} := (u_1^{bs})^{\otimes m} + \cdots + (u_r^{bs})^{\otimes m}
\ee
is the best rank-$r$ approximation for $\mc{F}$
%(or an almost best one if the best does not exist).
(if the best one does not exist, suppose
$\mc{X}^{bs}$ is sufficiently close to be best).
Let
\be \label{F=Xbs+E}
\mc{E} =  \mF - \mc{X}^{bs}
\quad \mbox{ and } \quad \eps = \| \mc{E} \|.
\ee
Let $A[\mc{F},\af],\,b[\mc{F},\af]$ be as in \reff{df:Ab[F,af]}.
They are linear in $\mF$. So, for all $\af \in \mathbb{B}_1$,
\be \label{Ab:F=Xls+E}
\left\{ \baray{rcr}
A[\mc{F},\af] &=& A[\mc{X}^{bs},\af] + A[\mc{E},\af], \\
b[\mc{F},\af] &=& b[\mc{X}^{bs},\af] + b[\mc{E},\af].
\earay \right.
\ee
Suppose all $(u_i^{bs})_1 \ne 0$.
Then we can scale them as
\[
u_i^{bs} = (\lmd_i^{bs})^{1/m} (1, v_i^{bs}),
\quad \lmd_i^{bs} \in \cpx, \quad v_i^{bs} \in \cpx^{\on}.
\]
%
%For convenience of notation, denote
%\[
%\overline{ v_i^{bs} } \, := \,
%\bbm (v_i^{bs})_1 & \cdots & (v_i^{bs})_{\bn} \ebm^T.
%\]
%
The tuple $(u_1^{bs},\ldots, u_r^{bs})$
is called {\it scaling-optimal} for $\mF$
if $(\lmd_1^{bs},\ldots, \lmd_r^{bs})$ is an optimizer of
the linear least squares problem
\be \label{opt-scal:vbs}
\min_{ (\lmd_1, \ldots, \lmd_r) \in \cpx^r } \quad
\| \lmd_1 (1,v_1^{bs})^{\otimes m} + \cdots +
\lmd_r (1,v_r^{bs})^{\otimes m} -\mc{F} \|^2.
\ee
When $\mc{X}^{bs}$ is the best rank-$r$ approximation,
the tuple $(u_1^{bs},\ldots, u_r^{bs})$ must be scaling-optimal.
So, it is reasonable to assume that $(u_1^{bs},\ldots, u_r^{bs})$
is scaling-optimal. Otherwise, we can always
replace it by a scaling-optimal one,
for the purpose of our approximation analysis.

Denote by $[ v ]_{\mathbb{B}_0}$
the vector of monomials in $\mathbb{B}_0$
evaluated at the point $v \in \cpx^{\bn}$.
The approximation quality of the low rank tensors
$\mc{X}^{gp},\mc{X}^{opt}$ produced by Algorithm~\ref{alg:tsr:aprox}
can be estimated as follows.

\begin{theorem} \label{thm:lrkapx:err}
Let $\hat{\xi}$ and $\mc{X}^{gp}$ be produced by Algorithm~\ref{alg:tsr:aprox},
and $\mc{F}, \mc{X}^{bs}, \mc{E}, u_i^{bs}, v_i^{bs}, \lmd_i^{bs}$ be as above.
Assume the conditions:
\bit

\item [i)] the vectors $[v_1^{bs}]_{\mathbb{B}_0},\ldots,
[v_r^{bs}]_{\mathbb{B}_0}$ are linearly independent;

\item [ii)] each matrix $A[\mF,\af] \,(\af \in \mathbb{B}_1)$ has full column rank;

\item [iii)] the tuple $(u_1^{bs},\ldots, u_r^{bs})$ is scaling-optimal for $\mF$;

\item [iv)] the matrix $L(\hat{\xi}) := \sum_{i=1}^{\bn}
\hat{\xi}_i M_{x_i} (G^{ls})$
%%given in \reff{M(xiWls)},
does not have a repeated eigenvalue.

\eit
If $\eps = \| \mc{E} \|$ is small enough, then
\be \label{err:F-Xrc:bs}
\| \mc{X}^{bs} - \mc{X}^{gp} \| = O(\eps) \, \mbox{ and } \,
 \| \mc{F} - \mc{X}^{opt} \| \leq \| \mc{F} - \mc{X}^{gp} \| = O(\eps),
\ee
where the constants in the above $O(\cdot)$ only depend on $\mF$.
\end{theorem}
\begin{proof} \,
By the condition i) and Theorem~\ref{thm:dstcV=>lmd},
there exists a generating matrix
$G^{bs} \in \cpx^{\mathbb{B}_0 \times \mathbb{B}_1}$ for $\mc{X}^{bs}$,
i.e., for all $\af \in \mathbb{B}_1$,
\[
\varphi[G^{bs},\af]( v_1^{bs} ) = \cdots =
\varphi[G^{bs},\af]( v_r^{bs} )  = 0
\, \mbox{ and } \,
A[\mc{X}^{bs},\af] \, G^{bs}(:,\af) = b[\mc{X}^{bs},\af].
\]
By \reff{Ab:F=Xls+E}, for all $\af \in \mathbb{B}_1$, we have
\[
\big\| A[\mc{F},\af]-A[\mc{X}^{bs},\af] \big\|_F \leq \eps
\, \mbox{ and } \,
\big\| b[\mc{F},\af]-b[\mc{X}^{bs},\af] \big\| \leq \eps.
\]
By the condition ii), the least square problem~\reff{ls:Aw=b:af}
has a unique minimizer, so there is no parameter
$\omega$ in \reff{omg:all}. Hence, we have
\[
G^{ls}(:,\af) = \big(A[\mc{F},\af]\big)^\dag b[\mc{F},\af]
\, \mbox{ and } \,
G^{bs}(:,\af) = \big(A[\mc{X}^{bs},\af]\big)^\dag b[\mc{X}^{bs},\af],
\]
where the superscript $^\dag$ denotes
the Moore-Penrose pseudoinverse~\cite{Demmel}.
When $\eps$ is sufficiently small, we have
\[
\| G^{ls}(:,\af) - G^{bs}(:,\af) \| = O(\epsilon).
\]
This follows from Theorem~3.4 of the book~\cite{Demmel}
or Theorem~3.9 of the book~\cite{SteSun90}.
The constant in the above $O(\epsilon)$ only depends on $\mF$.
Hence,
\[
\| G^{ls} - G^{bs} \|_F = O(\epsilon).
\]
Denote the matrix function in $\hat{\xi}$ and $G$
\[
L(\hat{\xi}, G) := \hat{\xi}_1 M_{x_1}(G) + \cdots + \hat{\xi}_{\on} M_{x_{\on}}(G).
\]
Then we can see that
\[
\| L( \hat{\xi},G^{bs}) - L( \hat{\xi},G^{ls}) \|_F
= \| \Sig_{i=1}^{\on} \hat{\xi}_i M_{x_i}(G^{bs}-G^{ls}) \|_F
\]
\[
\leq  \big(\Sig_{i=1}^{\on} |\hat{\xi}_i| \big)\,
\max_i \|  M_{x_i}(G^{bs}-G^{ls})  \|_F
\leq  \sqrt{\on} \|  G^{bs}-G^{ls}  \|_F = O(\epsilon).
\]
Recall that $\hat{\xi}$ is an eigenvector associated
with the least eigenvalue of $\hat{V}$ constructed for $G^{ls}$.
The last inequality follows from
$\Sig_{i=1}^{\on} |\hat{\xi}_i|  \leq \sqrt{\on} \| \hat{\xi} \|$
and the definition of $M_{x_i}(G)$, as in \reff{df:Mxi(W)}.
For each $i$, $(v_i^{bs})_1, \ldots, (v_i^{bs})_{\on}$
are respectively the eigenvalues of the companion matrices
\[
M_{x_1}(G^{bs}), \,  \ldots, \, M_{x_{\on}}(G^{bs}),
\]
with a common eigenvector \cite[\S2]{GPSTD}.
%
%By Stickelberger's Theorem (cf.~Sturmfels~\cite[Theorem~2.6]{Stu02}),
%
The eigenvalues of $L(\hat{\xi}, G^{bs})$ are
\[
\hat{\xi}^Tv_1^{bs}, \ldots, \hat{\xi}^T v_r^{bs}.
\]
By the condition iv), the matrix
$L(\hat{\xi}, G^{bs})$ does not have a repeated eigenvalue
when $\eps>0$ is small enough.
Recall that the unitary $\hat{Q}$ and upper triangular $\hat{T}$
are from the Schur decomposition~\reff{Schur:MWls} of
$L(\hat{\xi}, G^{ls}) = L(\hat{\xi})$ as in \reff{M(xiWls)}.
For $\eps >0$ small, there exist a unitary matrix $Q_1$
and an upper triangular matrix $T_1$ such that
$Q_1^* L( \hat{\xi}, G^{bs}) Q_1 = T_1$,
which satisfy
\be \label{SchurQTerr}
\| \hat{Q} - Q_1\|_F = O(\eps) \, \mbox{ and } \,
\| \hat{T} - T_1\|_F = O(\eps).
\ee
This can be derived from \cite{KPC94} or \cite[Theorem~4.1]{Sun95}.
The constants in the above $O(\eps)$
depend on $G^{ls}$, and then eventually only depend on $\mF$.

\,

The matrices $L(\hat{\xi}, G^{bs})$ and $T_1$ have common eigenvalues.
Because $T_1$ is upper triangular and its diagonal entries are all distinct,
there exists an upper triangular nonsingular matrix $R_1$ such that
$
\Lambda_1 \, := \, R_1^{-1} T_1 R_1
$
is diagonal. This results in the eigenvalue decomposition
\[
(Q_1R_1)^{-1} L( \hat{\xi}, G^{bs}) (Q_1R_1) = \Lambda_1.
\]
Since $v_1^{bs}, \ldots, v_r^{bs}$ are pairwise distinct,
which is implied by the condition i),
the polynomials $\varphi[G^{bs},\af]$ ($\af \in \mathbb{B}_1$)
do not have a repeated zero. So, the companion matrices
$
M_{x_1}(G^{bs}), \ldots, M_{x_{\on}}(G^{bs})
$
can be simultaneously diagonalized \cite[Corollary~2.7]{Stu02}.
There exist a nonsingular matrix $P$
and diagonal matrices $D_1,\ldots,D_n$ such that
\[
P^{-1} M_{x_1}(G^{bs}) P = D_1, \quad \ldots, \quad
P^{-1} M_{x_n}(G^{bs}) P = D_n,
\]
so that
\[
P^{-1} L(\hat{\xi}, G^{bs}) P = \sum_i \hat{\xi}_i D_i .
\]
Since $L(\hat{\xi}, G^{bs})$ does not have a repeated eigenvalue,
each eigenvector is unique up to scaling.
So, there exists a diagonal matrix $D_0$ such that
$ Q_1 R_1 = P D_0.$ For each $j$, the matrix
\[
\baray{rcl}
Q_1^* M_{x_j}(G^{bs}) Q_1
&=& R_1 D_0^{-1}  P^{-1} M_{x_j}(G^{bs}) P D_0 R_1^{-1} \\
&=& R_1 D_0^{-1} D_j D_0  R_1^{-1} = R_1   D_j  R_1^{-1}  =: \, S_j
\earay
\]
is upper triangular.
Note that $R_1D_jR_1^{-1}$ is an eigenvalue decomposition,
so that the diagonals of $D_j$ and $S_j$ are the same.
So, for all $i,j$, we have
\[
Q_1(:,i)^*M_{x_j}(G^{bs})Q_1(:,i) =
P^{-1}(i,:)M_{x_j}(G^{bs})P(:,i).
\]
($P^{-1}(i,:)$ denotes the $i$-th row of $P^{-1}$.)
For each $i$, the vector
\[
\bbm
P^{-1}(i,:)M_{x_1}(G^{bs})P(:,i) \\ \vdots \\ P^{-1}(i,:)M_{x_{\on}}(G^{bs})P(:,i))
\ebm
\]
is one of $v_1^{bs}, \ldots, v_r^{bs}$.
Up to a permutation of indices, we have
\[
v_i^{bs}  =
\bbm
Q_1(:,i)^*M_{x_1}(G^{bs})Q_1(:,i) \\ \vdots
\\ Q_1(:,i)^*M_{ x_{\bar{n}} }(G^{bs})Q_1(:,i))
\ebm.
\]
By \reff{hatv:ls:aprx} and \reff{SchurQTerr},
the above implies that (up to permutation of indices)
\be \label{vi:ls=bs+eps}
v_i^{ls} = v_i^{bs} + O(\eps) \, \,( i = 1,\ldots,r).
\ee
The constants in the above $O(\eps)$ eventually only depend on $\mF$.

\medskip

In Algorithm~\ref{alg:tsr:aprox},
$(\lmd_1^{ls}, \ldots, \lmd_r^{ls})$ is an optimizer of
\[
\min_{ (\lmd_1, \ldots, \lmd_r) \in \cpx^r }
\| \lmd_1 (1, v_1^{ls})^{\otimes m} + \cdots +
\lmd_r  (1, v_r^{ls})^{\otimes m} - \mF \|^2,
\]
while $(\lmd_1^{bs}, \ldots, \lmd_r^{bs})$ is an optimizer of
\[
\min_{ (\lmd_1, \ldots, \lmd_r) \in \cpx^r }
\| \lmd_1 (1, v_1^{bs})^{\otimes m} + \cdots +
\lmd_r  (1, v_r^{bs})^{\otimes m} - \mF \|^2,
\]
by the condition iii).
The tensors $(1,v_1^{bs})^{\otimes m}, \ldots,
(1,v_r^{bs})^{\otimes m}$ are linearly independent,
by the condition i). For $\eps>0$ small enough,
by \reff{vi:ls=bs+eps}, we can get
\[
 \|\lmd^{ls} - \lmd^{bs}\| = O(\eps)
\, \mbox{ and } \,
\|  \mc{X}^{gp}   -  \mc{X}^{bs}   \| = O(\eps).
\]
Hence,
\[
\|  \mc{F}   -  \mc{X}^{gp}   \|  \leq
\|  \mc{F}   -  \mc{X}^{bs} \| +
\|  \mc{X}^{bs}   -  \mc{X}^{gp}   \| = O(\eps).
\]
The constants in the above $O(\eps)$ eventually only depend on $\mF$.
Moreover, $\mc{X}^{opt}$ is improved from $\mc{X}^{gp}$
by solving the optimization problem \reff{nLS:lra:symF}.
So, $\|  \mc{F}   -  \mc{X}^{opt}   \|  \leq \|  \mc{F}   -  \mc{X}^{gp}   \|$,
and the proof is complete.
\end{proof}

In the condition i) of Theorem~\ref{thm:lrkapx:err},
we assume the linear independence of the vectors
$[v_1^{bs}]_{\mathbb{B}_0},\ldots, [v_r^{bs}]_{\mathbb{B}_0}$,
instead of $v_1^{bs} ,\ldots, v_r^{bs}$.
For $r>n$, it is still possible that
$[v_1^{bs}]_{\mathbb{B}_0},\ldots, [v_r^{bs}]_{\mathbb{B}_0}$
are linearly independent.
Algorithm~\ref{alg:tsr:aprox} would still produce
good low rank approximations when $r$ is big.
Please see the numerical experiments in Example~\ref{exmp:bigr}.
In the condition ii), each matrix $A[\mF,\af]$ has full column rank
only if its column number does not exceed the row number.
This is the case for all $A[\mF,\af]$ if and only if
$
r \leq \binom{n+m_1-1}{m_1}
$
where $m_1 = \left\lfloor \frac{m-1}{2} \right\rfloor$.
When $m=3,4$, the above requires $ r \leq n$;
when $m=5,6$, it requires $r \leq \binom{n+1}{2}$.
The constants hidden in $O(\cdot)$ of \reff{err:F-Xrc:bs}
are estimated in Example~\ref{emp:rler:LRsTA}.
In particular, if $\eps=0$, we can get $\mF = \mc{X}^{gp}$.

\begin{cor}[\cite{GPSTD}] \label{cor:apx=>TD}
Under the assumption of Theorem~\ref{thm:lrkapx:err},
if $\rank_S(\mc{F})=r$, then $\mc{X}^{gp}$ given by
Algorithm~\ref{alg:tsr:aprox} gives a Waring decomposition for $\mF$.
\end{cor}

%%%%%%%%%%%%%%%%%%%%%%%%%%%%%%%%
\iffalse

The assumptions in Theorem~\ref{thm:lrkapx:err} are often satisfied.
In contrast to the method in \cite{GPSTD},
Algorithm~\ref{alg:tsr:aprox} is more computationally tractable, because it
only requires to solve some linear least squares and Schur decompositions.
This is very efficient in applications, especially
for computing low rank decompositions for large tensors.
We refer to Example~\ref{sym:lwrk:STD}.

\fi
%%%%%%%%%%%%%%%%%%%%%%%%%%%

\section{Numerical experiments}
\label{sc:comp}
\setcounter{equation}{0}

In this section, we present numerical experiments
for computing low rank symmetric tensor approximations.
The computations were performed in MATLAB R2012a,
on a Lenovo Laptop with CPU@2.90GHz and RAM 16.0G.
In the Step~1 and Step~5 of Algorithm~\ref{alg:tsr:aprox},
the MATLAB function {\tt lsqnonlin} is used to solve
the nonlinear least squares problems
\reff{nls:omg:cmmu} and \reff{nLS:lra:symF}.
In the parameters for {\tt lsqnonlin},
{\tt MaxFunEvals} is set to be $10000$, and {\tt MaxIter} is $1000$.
The rank $r$ is chosen according to the first remark
in subsection~\ref{ssc:rks}.

We display a rank-$r$ tensor
$
\mc{X}= (u_1)^{\otimes m} + \cdots + (u_r)^{\otimes m}
$
by showing the decomposing vectors
$u_1,\dots, u_r$. Each $u_i$ row by row,
separated by parenthesises, with four decimal digits for each entry.

\subsection{Some examples}

\begin{exm}  \label{exmp-sin}
Consider the cubic tensor $\mF \in \mt{S}^3( \cpx^n )$ such that
\[
\mF_{i_1 i_2 i_3} = \sin( i_1 + i_2 + i_3).
\]
For $n=6$, the $3$ biggest singular values of
$\mbox{Cat}(\mF)$ are
\[
   5.7857, \quad   5.4357, \quad  7 \times 10^{-16}.
\]
%By Remark~\ref{est:rk:cat},
As in the subsection~\ref{ssc:rks},
we consider the rank-$2$ approximation.
When Algorithm~\ref{alg:tsr:aprox} is applied, we get the errors
\[
\| \mF - \mc{X}^{gp} \| \approx  5 \times 10^{-14},
\quad
\| \mF - \mc{X}^{opt} \| \approx  1 \times 10^{-15}.
\]
It took about $0.4$ second.
The computed rank-$2$ approximation $\mc{X}^{opt}$ is given as:
{\tiny
\begin{verbatim}
(0.7053-0.3640i   0.0748-0.7902i  -0.6245-0.4899i  -0.7496+0.2608i  -0.1856+0.7717i   0.5491+0.5731i),
(0.7053+0.3640i   0.0748+0.7902i  -0.6245+0.4899i  -0.7496-0.2608i  -0.1856-0.7717i   0.5491-0.5731i).
\end{verbatim} \noindent}The approximation tensors
$\mc{X}^{gp}$ and $\mc{X}^{opt}$ indeed
give a rank decomposition for $\mc{F}$, up to a tiny round-off error.
The computation is similar for other values of $n$.
%%%%% codes %%%%%%%%%%
\iffalse

clear all,
n = 6;
for i1 = 1 : n,   for i2 = 1 : n,  for i3 = 1 : n
           F(i1,i2,i3) = sin( i1 + i2 + i3 );
end,  end, end
%
CatMat = symtsr2CatalMat(F);
sv = svd(CatMat);
sv,
%
startime = tic;
[Xgr, err_gr, Xopt, err_opt] = SymLRTAGP(F, 2);
cptime = toc(startime);
[err_gr, err_opt],
cptime,
%
for k = 1 : size(Xopt,2)
Xgr(:,k).',
end
Xgr*diag( 1./Xgr(1,:) ),

\fi
%%%%%%%%%%%%%%%%%%%%%%%%
\end{exm}

\begin{exm} \label{exmp-rootsum}
Consider the quartic tensor $\mF \in \mt{S}^4( \cpx^5 )$ such that
\[
\mF_{i_1 i_2 i_3 i_4}  =   \sqrt{i_1+i_2+i_3+i_4}.
\]
The $5$ biggest singular values of $\mbox{Cat}(\mF)$ are respectively
\[
51.9534   ,\quad   0.9185   ,\quad   0.0133   ,\quad   0.0003  ,\quad  5.6 \times 10^{-6}.
\]
According to the subsection \ref{ssc:rks},
we consider rank-$4$ approximation and apply Algorithm~\ref{alg:tsr:aprox}.
It took about $1.8$ seconds. The approximation errors are:
\[
\|\mF-\mc{X}^{gp}\| \approx   0.1555, \quad
\|\mF-\mc{X}^{opt}\| \approx  0.0002.
\]
The tensor $\mc{X}^{opt}$ is given as:
{\tiny
\begin{verbatim}
(0.4075 + 0.4779i   0.2553 + 0.3523i   0.1582 + 0.2576i   0.0956 + 0.1869i   0.0563 + 0.1345i),
(0.2889 - 0.0030i   0.1598 + 0.0804i   0.0713 + 0.0886i   0.0165 + 0.0700i  -0.0089 + 0.0444i),
(0.8495 + 0.8887i   0.7904 + 0.8412i   0.7351 + 0.7961i   0.6836 + 0.7533i   0.6356 + 0.7127i),
(1.4501 + 0.0209i   1.4664 + 0.0191i   1.4828 + 0.0173i   1.4994 + 0.0154i   1.5162 + 0.0135i).
\end{verbatim}
}
%%%%%%%%%%%%%%%%%%%%%%%%%%%%%%%%%%%%%%%%%%%%%%%%%%%%%%%%%%%%%%%%%%%
\iffalse

clear all,
n = 5;  deg = 4;
for i1 = 1 : n, for i2 = 1 : n, for i3 = 1 : n, for i4 = 1 : n
  tensor(i1,i2,i3,i4) =  sqrt(i1+i2+i3+i4);
end, end, end, end
%
CatMat = symtsr2CatalMat(tensor);
sv = svd(CatMat);
sv',
%
startime = tic;
[Xgr4, gr_err4, Xopt4, opt_err4] = SymLRTAGP(tensor, 4);
cptime = toc(startime);
%
[gr_err4, opt_err4],
cptime,
%
for k = 1 : size(Xopt4,2)
 Xopt4(:,k).',
end
%
Xopt4*diag([1./Xopt4(1,:)]),

\fi
%%%%%%%%%%%%%%%%%%%%%%%%%%%%%%%%%%%%%%%%%%%%%%%%%%%%%%%%%%%%%%%%%%%%
\end{exm}

\begin{exm}
Consider the cubic tensor $\mF \in \mt{S}^3( \re^n )$ such that
\[
\mF_{i_1 i_2 i_3} =  i_1 + i_2 + i_3.
\]
The catalecticant matrix $\mbox{Cat}(\mF)$ has rank $2$.
The symmetric border rank of $\mF$ is $2$,
while its symmetric rank is $3$
(cf.~\cite[\S8.1]{CGLM08}). This is because
\[
\mF = a \otimes a \otimes b + a \otimes b \otimes a
+ b \otimes a \otimes a
= \lim_{\eps \to 0} \frac{1}{\eps}
\left( (a + \eps b)^{\otimes 3} - a^{\otimes 3} \right),
\]
where $a = (1, \ldots, 1)$ and $b=(1,2,\ldots, n)$.
The best rank-$2$ approximation does not exist,
but $\mF$ is arbitrarily close to a rank-$2$ tensor.
For convenience of discussion, consider the value $n=5$.
When $r=2$, the monomial sets $\bB_0, \bB_1$ are
\[
\bB_0 = \{1, x_1 \}, \quad
\bB_1 = \{ x_2, x_3, x_4, x_1^2, x_1x_2, x_1x_3, x_1x_4\}.
\]
The generating polynomials
$\varphi[G,\af](x)$ of the format \reff{vphi:W-af} are:
%\[
%\baray{c}
%-1+2x_1-x_2,\, -2+3x_1-x_3, \, -3+4x_1-x_4, \\
%-1+2x_1-x_1^2, \, -2+3x_1-x_1x_2,\, -3+4x_1-x_1x_3, \, -4+5x_1-x_1x_4.
%\earay
%\]
\[
\baray{ll}
\varphi[G,(0,1,0,0)](x) = -1+2x_1-x_2,
&\varphi[G,(0,0,1,0)](x) = -2+3x_1-x_3,  \\
\varphi[G,(0,0,0,1)](x) = -3+4x_1-x_4,
&\varphi[G,(2,0,0,0)](x) = -1+2x_1-x_1^2, \\
\varphi[G,(1,1,0,0)](x) = -2+3x_1-x_1x_2,
&\varphi[G,(1,0,1,0)](x) = -3+4x_1-x_1x_3, \\
\varphi[G,(1,0,0,1)](x) = -4+5x_1-x_1x_4.
&
\earay
\]
They have a repeated common zero,
which is $(1,1,1,1)$. The companion matrices
$M_{x_1}(G^{ls})$, $M_{x_2}(G^{ls})$, $M_{x_3}(G^{ls})$,
$M_{x_4}(G^{ls})$ are respectively
\[
\left[\baray{rr}  0 & -1  \\ 1  & 2 \earay\right],
\left[\baray{rr} -1 & -2  \\ 2  & 3 \earay\right],
\left[\baray{rr} -2 & -3  \\ 3  & 4 \earay\right],
\left[\baray{rr} -3 & -4  \\ 4  & 5 \earay\right].
\]
They have the repeated eigenvalue $1$ and are not diagonalizable.
We apply Algorithm~\ref{alg:tsr:aprox} to compute
its rank-$2$ approximation. It took about $0.3$ second. At Step~3,
the computed vectors $v_1^{ls}, v_2^{ls}$ are
{\tiny
\[
\bbm
  1.000000000000008 + 0.000000605979355i \\
  0.999999999999994 + 0.000001211958710i \\
  1.000000000000009 + 0.000001817938065i \\
  1.000000000000005 + 0.000002423917420i
\ebm, \quad
\bbm
  1.000000000000000 - 0.000000605979355i \\
  1.000000000000000 - 0.000001211958710i \\
  0.999999999999999 - 0.000001817938065i \\
  1.000000000000000 - 0.000002423917420i
\ebm.
\]
\noindent}They are quite close to $(1,1,1,1)$,
but not exactly same, because of round-off errors.
The resulting approximation errors are
\[
\| \mF - \mc{X}^{gp} \| \approx  8 \times 10^{-8}, \quad
\| \mF - \mc{X}^{opt} \| \approx  2 \times 10^{-9}.
\]
The approximating tensors $\mc{X}^{gp}$ and $\mc{X}^{opt}$ are almost same.
They are given as
{\tiny
\begin{verbatim}
(81.2272-46.8965i  81.2272-46.8964i  81.2273-46.8964i  81.2273-46.8963i  81.2273-46.8963i),
(81.2272+46.8965i  81.2272+46.8964i  81.2273+46.8964i  81.2273+46.8963i  81.2273+46.8963i).
\end{verbatim} \noindent}The
condition~iv) of Theorem~\ref{thm:lrkapx:err} is not satisfied.
%so it does not guarantee that
%$\mc{X}^{gp}$ is a good approximation.
However, in the computation, $\mc{X}^{gp}$ and $\mc{X}^{opt}$
are actually high quality rank-$2$ approximations.
%%%%% codes %%%%%%%%%%
%%%%%%%%%%%%%%%%%%%%%%%%%%%%%%%%%%%%%%%%%%%%%%%%%%%%%%%%%%%%%%%%%%%
\iffalse

clear all,
n = 5;
for i1 = 1 : n, for i2 = 1 : n, for i3 = 1 : n,
 F(i1,i2,i3) =  i1+i2+i3;
end, end, end
%
CatMat = symtsr2CatalMat(F);
sv = svd(CatMat);
sv',
startime = tic;
[Xgr2, gr_err2, Xopt2, opt_err2] = SymLRTAGP(F, 2);
cptime = toc(startime);
%
[gr_err2, opt_err2 ],
cptime,
%
for k= 1 : size(Xopt2, 2)
  Xopt2(:,k).',
end
%
Xgr2*diag([1./Xgr2(1,:)]),
Xopt2*diag([1./Xopt2(1,:)]),
%

\fi
%%%%%%%%%%%%%%%%%%%%%%%%%%%%%%%%%%%%%%%%%%%%%%%%%%%%%%%%%%%%%%%%%%%%
\end{exm}

\subsection{Approximation quality}

We present numerical experiments for exploring approximation qualities
of the low rank tensors $\mc{X}^{gp}$, $\mc{X}^{opt}$
given by Algorithm~\ref{alg:tsr:aprox}.
By Theorem~\ref{thm:lrkapx:err}, if $\mF  \in \mt{S}^m(\cpx^n)$
is sufficiently close to a rank-$r$ tensor, then $\mc{X}^{gp}$ and
$\mc{X}^{opt}$ are quasi-optimal rank-$r$ approximations.
First, we provide numerical experiments to estimate
the constants hidden in $O(\cdot)$ of \reff{err:F-Xrc:bs}.

\bcen
\begin{table}[htb]
\caption{The relative error bounds and computational time (in seconds)
for rank-$r$ approximations in $\mt{S}^3(\cpx^{10})$.}
{\small
\btab{ccrrrrrcllrc}  \specialrule{.2em}{0em}{0.1em}
\multirow{2}{*}{$r$}   & \multirow{2}{*}{$\eps$}   &  \multicolumn{5}{c}{ \tt err\text{-}gp } &  &
\multicolumn{2}{c}{ \tt err\text{-}opt }  & & {\tt time} \\  \cmidrule{3-7}  \cmidrule{9-10}  \cmidrule{12-12}
  &    & min & 25th &  50th & 75th & max  & & min & max  && average \\  \specialrule{.1em}{.1em}{0.1em}
\multirow{3}{*}{$1$}
& $10^{-2}$  & 0.961 & 0.974 & 0.979 & 0.983 & 0.997  & & 0.960 & 0.996  && 0.15   \\  \cmidrule{3-7}  \cmidrule{9-10} \cmidrule{12-12}
& $10^{-4}$  & 0.959 & 0.973 & 0.978 & 0.983 & 0.994  & & 0.956 & 0.993  && 0.14   \\  \cmidrule{3-7}  \cmidrule{9-10} \cmidrule{12-12}
& $10^{-6}$  & 0.954 & 0.974 & 0.978 & 0.985 & 0.992  & & 0.952 & 0.992  && 0.14 \\  \specialrule{.1em}{0.1em}{0.1em}
\multirow{3}{*}{$2$}
& $10^{-2}$  & 0.99  & 1.7  & 4.2  &  10.7 & 501.0 && 0.931  & 0.981  &&  0.23     \\  \cmidrule{3-7}  \cmidrule{9-10} \cmidrule{12-12}
& $10^{-4}$  & 0.95  & 1.6  & 3.3  & 13.4 & 329.9  && 0.917  & 0.979  &&  0.23    \\  \cmidrule{3-7}  \cmidrule{9-10} \cmidrule{12-12}
& $10^{-6}$  & 1.06  & 2.1  & 5.1  &  14.5 & 573.1 && 0.933  & 0.977  &&  0.23     \\  \specialrule{.1em}{.1em}{0.1em}
\multirow{3}{*}{$3$}
& $10^{-2}$  &  1.69  & 8.1  & 17.4 &  56.9 &  1732.0  && 0.907 &  0.962  &&  0.32      \\  \cmidrule{3-7}  \cmidrule{9-10} \cmidrule{12-12}
& $10^{-4}$  &  1.69  & 6.6  & 15.1 &  60.1 & 4787.9   && 0.904 &  0.969  &&  0.32    \\  \cmidrule{3-7}  \cmidrule{9-10} \cmidrule{12-12}
& $10^{-6}$  &  1.17  & 7.1  & 18.3 &  63.1 & 1152.4   && 0.915 &  0.966  &&  0.31     \\  \specialrule{.1em}{.1em}{.1em}
\multirow{3}{*}{$4$}
& $10^{-2}$  &  1.38  & 5.9  & 12.5 &  41.2 &  620.8  &&  0.875  &  0.953  &&  0.40     \\  \cmidrule{3-7}  \cmidrule{9-10} \cmidrule{12-12}
& $10^{-4}$  &  1.76  & 6.5  & 15.5 & 33.8 &  8374.7  &&  0.879  &  0.946  &&  0.40   \\  \cmidrule{3-7}  \cmidrule{9-10} \cmidrule{12-12}
& $10^{-6}$  &  1.72  & 5.8  & 13.3 & 55.7 &  9049.8  &&  0.882  &  0.950  &&  0.39    \\  \specialrule{.1em}{.1em}{0.1em}
\multirow{3}{*}{$5$}
&  $10^{-2}$ &  2.45  & 7.7  &  18.0 &  68.8 & 5162.0  && 0.838 &  0.927   && 0.49  \\  \cmidrule{3-7}  \cmidrule{9-10} \cmidrule{12-12}
& $10^{-4}$  & 1.75  &  10.1 &  24.9 &  69.7 & 7316.6  && 0.835 &  0.933   && 0.48   \\   \cmidrule{3-7}  \cmidrule{9-10} \cmidrule{12-12}
& $10^{-6}$  &  3.51  & 11.4 &  22.4 &  48.6 & 4483.6  && 0.849 &  0.926   && 0.47   \\
\specialrule{.2em}{.1em}{0em}
\etab
}
\label{tb:rlerr:cubic}
\end{table}
\ecen

\bcen
\begin{table}[htb]
\caption{The relative error bounds and computational time (in seconds)
for rank-$r$ approximations in $\mt{S}^4(\cpx^{10})$.}
{\small
\btab{ccrrrrrcllrc} \specialrule{.2em}{0em}{0.1em}
\multirow{2}{*}{$r$}   & \multirow{2}{*}{$\eps$}   &  \multicolumn{5}{c}{ \tt err\text{-}gp }  & &
\multicolumn{2}{c}{ \tt err\text{-}opt } && {\tt time}   \\  \cmidrule{3-7}  \cmidrule{9-10}  \cmidrule{12-12}
  &    &  min & 25th & 50th & 75th & \quad max  & & min & max  && {average} \\  \specialrule{.1em}{.1em}{0.1em}
\multirow{3}{*}{$1$}
& $10^{-2}$  & 0.988 & 0.992 & 0.994 & 0.996 & 0.998 && 0.988 & 0.998  && 0.66    \\  \cmidrule{3-7}  \cmidrule{9-10} \cmidrule{12-12}
& $10^{-4}$  & 0.988 & 0.993 & 0.995 & 0.996 & 0.999 && 0.988 & 0.999  && 0.67     \\  \cmidrule{3-7}  \cmidrule{9-10}  \cmidrule{12-12}
& $10^{-6}$  & 0.988 & 0.992 & 0.995 & 0.996 & 0.999 && 0.988 & 0.999  && 0.66    \\  \specialrule{.1em}{.1em}{0.1em}
\multirow{3}{*}{$2$}
  & $10^{-2}$  & 0.99 & 1.2 &  1.7  &  5.2  & 1054.7 && 0.980 &  0.995  &&  1.12    \\  \cmidrule{3-7}  \cmidrule{9-10} \cmidrule{12-12}
  & $10^{-4}$  & 0.99 & 1.1  &  1.8  &  5.3  & 1453.1  && 0.981 & 0.997  && 1.02   \\  \cmidrule{3-7}  \cmidrule{9-10}  \cmidrule{12-12}
  & $10^{-6}$  & 0.99  & 1.1  &  1.9  &  6.3  &  688.6 && 0.980 & 0.996  && 1.01    \\  \specialrule{.1em}{.1em}{0.1em}
\multirow{3}{*}{$3$}
  & $10^{-2}$  & 1.05  & 3.2  & 8.4  &  20.5  &  358.6 && 0.975 &  0.992  && 1.47     \\  \cmidrule{3-7}  \cmidrule{9-10} \cmidrule{12-12}
  & $10^{-4}$  &  1.25  & 5.5 & 13.0  &  45.9  &  2418.3  && 0.975 & 0.990 && 1.45    \\  \cmidrule{3-7}  \cmidrule{9-10} \cmidrule{12-12}
  & $10^{-6}$  & 1.01  & 2.5   & 8.4  & 23.6  & 1651.6  && 0.972 &  0.993  && 1.45     \\ \specialrule{.1em}{.1em}{0.1em}
\multirow{3}{*}{$4$}
  & $10^{-2}$  &  1.19  & 2.9  & 7.6  & 19.1 & 5722.1 &&  0.965 &  0.992 &&   1.94   \\  \cmidrule{3-7}  \cmidrule{9-10} \cmidrule{12-12}
  & $10^{-4}$  &  1.18  & 2.8  & 5.8  & 17.8  & 3132.5 &&  0.965 & 0.988 &&  1.92   \\  \cmidrule{3-7}  \cmidrule{9-10} \cmidrule{12-12}
  & $10^{-6}$  & 1.16 & 2.3  & 4.8  & 14.4  &  549.6 &&  0.964 & 0.990  &&  1.92   \\   \specialrule{.1em}{.1em}{0.1em}
\multirow{3}{*}{$5$}
  &  $10^{-2}$ & 1.74  &  3.9 &  8.3  &  22.9  & 3985.6 && 0.961 & 0.984   &&  2.40   \\   \cmidrule{3-7}  \cmidrule{9-10} \cmidrule{12-12}
  & $10^{-4}$  &  1.36 & 3.8  &  8.2 & 20.5  &   1613.7  && 0.959 & 0.983  &&  2.38  \\   \cmidrule{3-7}  \cmidrule{9-10} \cmidrule{12-12}
  & $10^{-6}$  & 1.26  & 5.5   & 12.3 & 28.1  &  1876.8  && 0.961 & 0.988  &&  2.37   \\
  \specialrule{.2em}{.1em}{0em}
\etab
}
\label{tb:err:quar}
\end{table}
\ecen

\begin{exm}  \label{emp:rler:LRsTA}
Generate rank-$r$ tensors of the form
\[
\mc{R} =  (u_1)^{\otimes m} + \cdots + (u_r)^{\otimes m},
\]
where each $u_i \in \cpx^n$ has random real and imaginary parts,
obeying Gaussian distributions.
Then, choose a random tensor $\mc{E} \in \mt{S}^m(\cpx^n)$,
whose entries are all randomly generated.
Scale $\mc{E}$ to have a desired norm $\eps >0$. Let
\[
\mF = \mc{R} + \mc{E}.
\]
Approximation qualities of $\mc{X}^{gp}$ and $\mc{X}^{opt}$
can be measured by the relative errors
\[
{\tt err\text{-}gp} :=  \| \mF - \mc{X}^{gp} \| \,\, / \,\, \eps,
\quad
{\tt err\text{-}opt} :=  \| \mF - \mc{X}^{opt} \| \,\, / \,\, \eps.
\]
They can be used to estimate the constants
hidden in $O(\cdot)$ of \reff{err:F-Xrc:bs} in Theorem~\ref{thm:lrkapx:err}.
We choose the values of $n,m,r, \eps$ as
\[
n=10, \quad m=3,4, \quad
r = 1, 2, 3, 4, 5
\quad \mbox{ and } \quad
\eps = 10^{-2}, 10^{-4}, 10^{-6}.
\]
For each $(n,m,r,\eps)$, we generate $100$ random instances of $\mF$.
For each instance, apply Algorithm~\ref{alg:tsr:aprox} to compute the approximating
tensors $\mc{X}^{gp}$ and $\mc{X}^{opt}$,
whose symmetric ranks $\leq r$.
For these $100$ instances, we record the
minimum, maximum and quartiles of
%the $25$th smallest, the $50$th smallest,  the $75$th smallest,
{\tt err\text{-}gp}. For {\tt err\text{-}opt},
we record the minimum and maximum,
because the variance is small.
The consumed computational time for each instance
also does not vary much, so we record the average (in seconds).
For the case of approximations in $\mt{S}^3(\cpx^{10})$,
the computational results are reported in Table~\ref{tb:rlerr:cubic}.
For tensors in $\mt{S}^4(\cpx^{10})$, the results
are reported in Table~\ref{tb:err:quar}.
In these two tables, the first column lists the values of $r$;
the second column lists the values of $\eps$;
the columns entitled by {\tt err\text{-}gp}
list the relative error bounds {\tt err\text{-}gp}
for the minimum, maximum and quartiles;
the columns entitled by {\tt err\text{-}opt}
list the relative error bounds {\tt err\text{-}opt}
for the minimum and maximum;
the last column lists the average of computational time
of Algorithm~\ref{alg:tsr:aprox}.
As one can see, for the majority of instances,
$\mc{X}^{gp}$ is a good approximation.
For a few cases, the relative error {\tt err\text{-}gp} is big,
but the absolute error $\| \mF - \mc{X}^{gp} \|$ is still small
(compared with the tensor norm $\| \mF \|$).
For all instances, the improved tensor
$\mc{X}^{opt}$ (by solving the optimization \reff{nLS:lra:symF}
with $\mc{X}^{gp}$ as a starting point)
gives a high quality rank-$r$ approximation.
\end{exm}

A traditional approach for computing low rank tensor approximations
is to solve \reff{nLS:lra:symF} by nonlinear least squares (NLS) methods.
NLS requires a starting point for the approximating tensor.
A typical method for choosing starting points is to flatten
$\mF \in \mt{S}^m(\cpx^n)$ into a cubic tensor
$\widetilde{\mF} \in \mt{S}^{k_1}(\cpx^n) \otimes \mt{S}^{k_2}(\cpx^n) \otimes \cpx^n$,
where $k_1+k_2+1=m$ and $k_1=k_2$ ($m$ is odd) or $k_1=k_2+1$ ($m$ is even).
When $r \leq \binom{n-1+\lfloor(m-1)/2\rfloor}{n-1}$ and $\mF$ is rank-$r$,
the methods in \cite{DDeLa14,LRA93} can be applied to get
a decomposition for $\widetilde{\mF}$.
Low rank tensors often has unique rank decompositions,
so the computed decomposition often corresponds
to the unique symmetric decomposition.
For low rank approximations, this approach can still be applied to
get an approximate Waring decomposition,
which then can be used a starting point for solving \reff{nLS:lra:symF}
by NLS. This is the current state-of-the-art.
However, when $r > \binom{n-1+\lfloor(m-1)/2\rfloor}{n-1}$,
there are no general methods for choosing good starting points.
In computational practice, people often choose random ones.
On the other hand, Algorithm~\ref{alg:tsr:aprox} does not depend
on the choice of starting points. It has good performance even if
the value of $r$ is large. The following
is a computational experiment for this.

\begin{exm} \label{exmp:bigr}
We compare Algorithm~\ref{alg:tsr:aprox}
with the NLS approach for solving \reff{nLS:lra:symF}, i.e.,
only Step~5 of Algorithm~\ref{alg:tsr:aprox} is implemented
with randomly chosen starting points.
As suggested by a referee, we make the comparison
on low rank tensors of the form
\[
\mc{R} =  \tau (u_1)^{\otimes m} + \tau^2 (u_1)^{\otimes m}
+\cdots + \tau^r (u_r)^{\otimes m},
\]
where $u_1, \ldots, u_r$ are random complex vectors
as in Example~\ref{emp:rler:LRsTA} and  $\tau >0$ is a scaling factor.
In the numerical test, we choose
\[
r > \binom{n-1+ \lfloor (m-1)/2 \rfloor}{ n-1}
\quad \mbox{ and } \quad
\tau = 1000^{\frac{1}{r}}.
\]
We choose $\mc{E}$, with $\| \mc{E} \| = \eps$,  in the same way as
in Example~\ref{emp:rler:LRsTA} and let
$\mF = \mc{R} + \mc{E}$.
Let $\mc{X}^{opt}$ be the low rank tensor
produced by Algorithm~\ref{alg:tsr:aprox} and the relative error
{\tt err\text{-}opt} is measured in the same way.
For the NLS method, we apply the MATLAB function {\tt lsqnonlin}
with random starting points.
For a better chance of success,
we apply {\tt lsqnonlin} $10$ times
with different random starting points,
and then select the best low rank approximating tensor that is found,
which we denote as $\mc{X}^{nls}$. Its relative error is similarly measured as
\[
{\tt err\text{-}nls} :=  \| \mF - \mc{X}^{nls} \| \,\, / \,\, \eps.
\]
We test for the values $n=10,m=3,$
\[
r = 11, 12, 13, 14, 15
\quad \mbox{ and } \quad
\eps = 10^{-2}, 10^{-4}, 10^{-6}.
\]
Use {\tt tm-opt} to denote the time consumed by
Algorithm~\ref{alg:tsr:aprox}, 
and use {\tt tm-nls} to denote the
time of applying {\tt lsqnonlin} $10$ times in total.
The ratio {\tt err\text{-}nls}/{\tt err\text{-}opt}
measures the difference of their approximation qualities.
The bigger it is,
the better the approximation quality $\mc{X}^{opt}$ has.
The ratio {\tt tm-nls}/{\tt tm-opt}
measures the difference of their computational time.
The bigger it is, the more expensive the NLS is.
For each $(r,\eps)$ as above, we generate $20$ random instances
of $\mF$, and then compute $\mc{X}^{opt}$, $\mc{X}^{nls}$
respectively by applying Algorithm~\ref{alg:tsr:aprox} and {\tt lsqnonlin}.
We report the {\tt 1st}, {\tt 5th, 10th, 15th} and {\tt 20th}
smallest values of the ratio {\tt err\text{-}nls}/{\tt err\text{-}opt}
(the {\tt 1st} one is the minimum and the {\tt 20th} one is the maximum),
and the minimum, medium and the maximum values of
the ratio {\tt tm-nls}/{\tt tm-opt}.
These ratios are reported in Table~\ref{tb:cubic:scaled}.
The first column lists the values of $r$;
the second column lists the values of $\eps$;
the third through seventh columns
list the values of {\tt err\text{-}nls}/{\tt err\text{-}opt};
the last three columns list the values of
{\tt tm-nls}/{\tt tm-opt}.
As one can see, the advantage of Algorithm~\ref{alg:tsr:aprox}
over NLS is quite clear,
in terms of both the approximation quality and the computational time.
There are a few cases that {\tt err\text{-}nls}/{\tt err\text{-}opt} is small.
This is because the optimization problem \reff{nls:omg:cmmu}
was not solved successfully and $\mc{X}^{gp}$
is not an accurate estimate for the best low rank approximation.
\end{exm}

\bcen
\begin{table}[htb]
\caption{Comparison between Algorithm~\ref{alg:tsr:aprox}
and the NLS for computing low rank tensor approximations in $\mt{S}^3(\cpx^{10})$.}
{\small
\btab{ccrcccccrrr}  \specialrule{.2em}{0em}{0.1em}
\multirow{2}{*}{$r$}   & \multirow{2}{*}{$\eps$}   &
\multicolumn{5}{c}{ \tt err\text{-}nls/err\text{-}opt} &  &
\multicolumn{3}{c}{ \tt tm\text{-}nls/tm\text{-}opt}   \\
\cmidrule{3-7}  \cmidrule{9-11}
  &    & 1st & 5th &  10th & 15th & 20th  && min & med. & max   \\  \specialrule{.1em}{.1em}{0.1em}
\multirow{3}{*}{$11$}
& $10^{-2}$  & 1.0  & 1.0  &  $0.3 \cdot 10^4$ &  $0.5 \cdot 10^4$ &  $1.2 \cdot 10^4$
              & &  3.3 & 16.1  & 27.7 \\  \cmidrule{3-7}  \cmidrule{9-11}
& $10^{-4}$  & 1.0  & 1.0  & $2.2 \cdot 10^5$ &  $4.9 \cdot 10^5$ &  $9.5 \cdot10^5$
              & &  1.7   &  7.3   & 30.1  \\  \cmidrule{3-7}  \cmidrule{9-11}
& $10^{-6}$  & 0.7  & $0.3 \cdot 10^8$  & $0.4\cdot 10^8$ &   $0.7 \cdot 10^8$ &   $1.4 \cdot 10^8$
             & &   3.0  &  11.7  &   31.2 \\  \specialrule{.1em}{.1em}{0.1em}
\multirow{3}{*}{$12$}
& $10^{-2}$  & 1.0  & 1.0 &  $0.4 \cdot 10^4$  &  $0.7 \cdot 10^4$  &  $1.1 \cdot 10^4$
             & &  1.6  & 2.4  &  4.1\\  \cmidrule{3-7}  \cmidrule{9-11}
& $10^{-4}$  & 1.0  &   $0.4 \cdot 10^6$  & $0.5 \cdot 10^6$  &  $0.9 \cdot 10^6$  & $1.0 \cdot 10^6$
             & &  2.3 & 2.5  &  10.7  \\  \cmidrule{3-7}  \cmidrule{9-11}
& $10^{-6}$  & 1.0  & $0.2\cdot 10^8$ &  $0.4 \cdot 10^8$ &  $0.6 \cdot 10^8$  &  $1.3 \cdot 10^8$
             & & 2.1 & 2.9 & 10.6 \\  \specialrule{.1em}{.1em}{0.1em}
\multirow{3}{*}{$13$}
& $10^{-2}$  & 1.0  & 1.3  &  $0.6 \cdot 10^4$  &  $0.7 \cdot 10^4$  &  $1.1 \cdot 10^4$
             & &   1.1 & 1.8 & 4.5  \\  \cmidrule{3-7}  \cmidrule{9-11}
& $10^{-4}$  & 0.8  & $0.4 \cdot 10^6$ &   $0.5 \cdot 10^6$ &  $0.7 \cdot 10^6$  & $1.7 \cdot 10^6$
             & & 1.4  &   1.8  &  2.7 \\  \cmidrule{3-7}  \cmidrule{9-11}
& $10^{-6}$  & 1.0  & $0.3 \cdot 10^8$ &  $0.6 \cdot 10^8$ &  $0.8 \cdot 10^8$ &  $1.2 \cdot 10^8$
             & &  1.5  &  1.8  &   2.5 \\  \specialrule{.1em}{.1em}{0.1em}
\multirow{3}{*}{$14$}
& $10^{-2}$  & 2.1  &  $0.3  \cdot 10^4$  &  $0.7 \cdot 10^4$  &  $1.0 \cdot 10^4$  & $1.6   \cdot 10^4$
& &   1.2 &  1.4 &  2.5 \\  \cmidrule{3-7}  \cmidrule{9-11}
& $10^{-4}$  & 0.5  &  1.6  &  $0.4 \cdot 10^6$  &   $0.8  \cdot 10^6$  &  $1.3 \cdot 10^6$
             & &   0.9  &   1.4 &   1.8\\  \cmidrule{3-7}  \cmidrule{9-11}
& $10^{-6}$  & 0.9  & 4.3  &  $0.3 \cdot 10^8$ &  $0.6  \cdot 10^8$  & $1.2 \cdot 10^8$
             & &  1.2  &   1.3 & 12.8  \\  \specialrule{.1em}{.1em}{0.1em}
\multirow{3}{*}{$15$}
& $10^{-2}$  & 0.7  & 1.4  & 3.0  &   $0.6 \cdot 10^4$  & $1.4 \cdot 10^4$
             & &   0.8  & 1.1  &  1.6 \\  \cmidrule{3-7}  \cmidrule{9-11}
& $10^{-4}$  & 0.4  & 1.6 &  $0.4 \cdot 10^6$  &  $0.8 \cdot 10^6$  & $1.7  \cdot 10^6$
             & & 0.8  &  1.0  &    1.6 \\  \cmidrule{3-7}  \cmidrule{9-11}
& $10^{-6}$  & 0.6  & 3.6  &  $0.5  \cdot 10^8$  &  $1.1  \cdot 10^8$ &  $2.0 \cdot 10^8$
             & & 1.0  & 1.1  &  1.7 \\
\specialrule{.2em}{0em}{0.1em}
\etab
}
\label{tb:cubic:scaled}
\end{table}
\ecen

\subsection{Waring decompositions}

As mentioned in Corollary~\ref{cor:apx=>TD},
Algorithm~\ref{alg:tsr:aprox} can also be applied
to compute Waring decompositions,
when the tensor $\mF$ has rank $r$, under suitable conditions.
For generic tensors of certain ranks,
the Waring decomposition is unique,
as shown in \cite{ChOtVan15,GalMel}.
When it is unique, the Waring decomposition
can be computed by Algorithm~\ref{alg:tsr:aprox}.
When it is not unique, Algorithm~\ref{alg:tsr:aprox}
can be applied to get distinct Waring decompositions,
if different starting points are used
for solving \reff{nls:omg:cmmu} in the Step~1.
The following is such an example.

\begin{exm}
Consider the tensor $\mF \in \mt{S}^4(\cpx^4)$ that is given as
\[
\text{ \tiny
$\left[\baray{r}     1  \\   1  \\   1  \\   1  \earay \right]^{\otimes 4} +
\left[\baray{r}     1  \\   1  \\   2  \\  -3  \earay \right]^{\otimes 4} +
\left[\baray{r}     1  \\   2  \\  -3  \\   1  \earay \right]^{\otimes 4} +
\left[\baray{r}     1  \\  -3  \\   2  \\   1  \earay \right]^{\otimes 4} +
\left[\baray{r}     1  \\  -1  \\   3   \\  2  \earay \right]^{\otimes 4} +
\left[\baray{r}     1  \\   2  \\  -1   \\  3  \earay \right]^{\otimes 4} +
\left[\baray{r}     1  \\   3  \\  -1   \\  2  \earay \right]^{\otimes 4} +
\left[\baray{r}     1  \\   1  \\   2   \\  3  \earay \right]^{\otimes 4}.
$
}
\]
Clearly, its rank is at most $8$.
Indeed, the rank equals $8$, because $\rank\, \mbox{Cat}(\mF) = 8$,
which is a lower bound for $\rank_S(\mF)$ by \reff{rk:Cat<=B<=S}.
We apply Algorithm~\ref{alg:tsr:aprox}
with $r=8$ to compute its Waring decomposition.
In the Step~1, the solution to the least squares problem~\reff{ls:Aw=b}
is not unique, so the optimization problem \reff{nls:omg:cmmu}
needs to be solved. For this tensor, the optimizer $G^{ls}$
of \reff{nls:omg:cmmu} is not unique.
In the Step~1 of Algorithm~\ref{alg:tsr:aprox},
when random starting points are used, we can get two optimizers,
which gives two distinct Waring decompositions.
In addition to the one given above, we get the second
decomposition $\sum_{i=1}^8 u_i^{\otimes 4}$
where the vectors $u_i$ are:
{\tiny
\begin{verbatim}
(0.9702   -3.0822    2.1701    1.0051),
(1.0296    2.2812   -3.0472    1.1889),
(0.8181    0.4112    2.3661    2.6623),
(1.1530   -0.4734    2.7425    2.1675),
(1.2091    2.9379   -0.4232    2.6333),
(0.7692    1.4901   -0.8059    2.9758),
(0.6704   -0.4632   -0.5726    0.4538),
(1.0208    1.0288    2.0039   -2.9538).
\end{verbatim} \noindent}Chiantini~et al.~\cite{ChOtVan15}
showed that a generic tensor of rank $8$ in $\mt{S}^4(\cpx^4)$
has two distinct Waring decompositions.
The above computation confirmed this fact.
%%%%%%%%%%%%%%%%%%%%%%%%%%%%%
\iffalse

clear all, clc,
m = 4; n = 4; r = 8;
U = [...
1   1   1   1   1   1   1   1 ; ...
1   1   2  -3  -1   2   3   1 ; ...
1   2  -3   2   3  -1  -1   2 ; ...
1  -3   1   1   2   3   2   3 ; ...
];
calA = hmg_pwlist(n,m);
vF = eval_monvec(calA, U(:,1) );
for s = 2 : r
   vF = vF + eval_monvec(calA, U(:,s) );
end
F = vec2symtsr(vF,n,m);
%
startime = tic;
repeat = 0;
while (repeat < 5 )
repeat = repeat + 1;
[Vgr, err_gr, Xopt, opt_err] = SymLRTAGP(F,r);
U,
%Vgr*diag(1./Vgr(1,:) ),
%real( Xopt*diag(1./Xopt(1,:) )    ),
real( Xopt.' ),
opt_err,
end

CM = symtsr2CatalMat(F);
rank(CM),

\fi
%%%%%%%%%%%%%%%%%%%%%%%%%%%%
\end{exm}

In the following, we report more numerical experiments
for using Algorithm~\ref{alg:tsr:aprox} to compute Waring decompositions.

\begin{exm} \label{sym:lwrk:STD}
%(Decompositions of low rank symmetric tensors.)
%By Corollary~\ref{cor:apx=>TD}, if $\rank_S(\mF) =r$, then $\mc{X}^{gp}$
%produced by Algorithm~\ref{alg:tsr:aprox} gives a rank decomposition for $\mF$.
%We report numerical experiments for decomposing random low rank tensors.
As in Example~\ref{emp:rler:LRsTA},
we generate $\mF$ in the same way, except letting $\mc{E}=0$.
For each instance of $\mF$, we apply
Algorithm~\ref{alg:tsr:aprox} to get $\mc{X}^{gp}$.
Because of round-off errors and numerical issues, $\mc{X}^{gp}$ may not give an
exact decomposition for $\mF$. However, the improved tensor
$\mc{X}^{opt}$ usually gives a more accurate decomposition.
Generally, $\mc{X}^{opt}$ can be computed easily,
because $\mc{X}^{gp}$ is already very accurate.
For each $(n,m,r)$ in Table~\ref{tab:time:lwrkSTD},
we generate $100$ random instances of $\mF$, except for
\[
(n,m,r) \in \{ (40,3,30), \, (25,4,25), \, (30, 4, 30) \}.
\]
(For the above $3$ cases, only $10$ random instance were generated,
because of the relatively long computational time.)
For each $(n,m,r)$ with the symbol $\ast$, it means that
the least squares problem \reff{ls:Aw=b} has a parameter $\omega$
in its optimal solution, and the nonlinear optimization problem
\reff{nls:omg:cmmu} needs to be solved.
For all the instances, we get correct rank decompositions (up to tiny round-off errors).
%
%Since $\mc{X}^{gp} = \mF$, Step~5 of Algorithm~\ref{alg:tsr:aprox}
%does not need to be performed.
%
For each $(n,m,r)$, the average of computational time (in seconds)
is reported in Table~\ref{tab:time:lwrkSTD}. As we can see,
Waring decompositions of low rank symmetric tensors
can be computed efficiently by Algorithm~\ref{alg:tsr:aprox}.
%
%Algorithm~\ref{alg:tsr:aprox} only needs to solve some linear least squares
%and Schur's decompositions. So, it is suitable
%for computing low rank symmetric tensor decompositions,
%especially for large tensors.
%
\end{exm}

\bcen
\begin{table}[htb]
\caption{Computational time (in seconds) for
computing decompositions of rank-$r$ symmetric tensors in $\mt{S}^m(\cpx^n)$.}
\btab{lrclrclr}  \specialrule{.2em}{0em}{0.1em}
$(n,m,r)$  & {\tt time} & \quad & $(n,m,r)$ & {\tt time} & \quad & $(n,m,r)$ & {\tt time} \\
 \cmidrule{1-2}  \cmidrule{4-5}  \cmidrule{7-8}
(4, 3, 5)$\ast$    &  0.42    && (3, 4, 5)$\ast$  & 0.08  &&  (4, 5, 10)  &   0.24   \\
 \cmidrule{1-2}  \cmidrule{4-5}  \cmidrule{7-8}
(8, 3, 10)$\ast$   & 12.73    & &  (5, 4, 10)$\ast$  & 0.90  &&  (6, 5, 20)  &   2.46   \\
 \cmidrule{1-2}  \cmidrule{4-5}  \cmidrule{7-8}
(14, 3, 15)$\ast$   & 74.23    & & (8, 4, 15)$\ast$  &  6.47 &&  (7, 5, 30)$\ast$  &   76.50   \\
 \cmidrule{1-2}  \cmidrule{4-5}  \cmidrule{7-8}
(20, 3, 20)  &    16.67  && (10, 4, 20)$\ast$  &   74.69  &&  (4, 6, 10)  &  0.97 \\
 \cmidrule{1-2}  \cmidrule{4-5}  \cmidrule{7-8}
(30, 3, 25)  &    95.01  &&  (25, 4, 25)  & 528.18   &&    (5, 6, 20)  &  4.18   \\
 \cmidrule{1-2}  \cmidrule{4-5}  \cmidrule{7-8}
(40, 3, 30)  &   410.74  &&  (30, 4, 30)  & 1483.65  && (6, 6, 30)$\ast$  &   25.08  \\
\specialrule{.2em}{.1em}{0em}
\etab
\label{tab:time:lwrkSTD}
\end{table}
\ecen

\section{Conclusions and Future Work}
\label{sc:confu}

This paper studies the low rank approximation problem
for symmetric tensors. The main approach is to use generating polynomials.
The method is described in Algorithm~\ref{alg:tsr:aprox}.
We showed that if a symmetric tensor is sufficiently close to
a low rank one, the low rank approximating tensors produced by
Algorithm~\ref{alg:tsr:aprox} are quasi-optimal.
Moreover, Algorithm~\ref{alg:tsr:aprox} can also be applied
to compute Waring decompositions. Numerical experiments
for the computation are also given.

There is still much future work to do for computing low rank
symmetric tensor approximations.
When is Algorithm~\ref{alg:tsr:aprox} able to produce
best low rank approximations? If it does, how can we detect that
the computed low rank approximation is the best? Mathematically,
the best low rank approximation might not exist \cite{DeSLim08}.
In such a case, how can we get a low rank approximation
that is close to being best?
If a symmetric tensor is not close to a low rank one,
Algorithm~\ref{alg:tsr:aprox} can be still be
applied to get a low rank approximation,
but its quality cannot be guaranteed.
For such a case, how can get better low rank approximations?
To the best of the author's knowledge,
these questions are mostly open.
They are interesting questions for future work.

\bigskip
\noindent
{\bf Acknowledgement}
The research was partially supported by the NSF grants
DMS-1417985 and DMS-1619973.
The author would like to thank the anonymous referees
for fruitful suggestions on improving the paper.

\end{document}